\documentclass[a4paper,10pt]{elsarticle}

\usepackage{latexsym}
\usepackage{amssymb}
\usepackage{amsthm}
\usepackage{amsfonts}
\usepackage{amsmath}
\usepackage{indentfirst}
\usepackage{graphicx,subfigure}
\usepackage{mathtools}
\usepackage[usenames,dvipsnames]{xcolor}

\usepackage{tikz}
\usetikzlibrary{intersections,patterns,arrows,decorations.pathreplacing}
\usepackage{hyperref}
\usepackage{epic, eepic}

\newtheorem{theorem}{Theorem}[section]

\newtheorem{corollary}[theorem]{Corollary}
\newtheorem{lemma}[theorem]{Lemma}
\newtheorem{proposition}[theorem]{Proposition}
\newtheorem{conjecture}[theorem]{Conjecture}
\newtheorem{problem}[theorem]{Problem}

\theoremstyle{definition}
\newtheorem{definition}[theorem]{Definition}
\newtheorem{remark}[theorem]{Remark}
\newtheorem{question}[theorem]{Question}
\newtheorem{example}[theorem]{Example}
\newtheorem{algorithm}[theorem]{Algorithm}

\numberwithin{equation}{section}
\numberwithin{figure}{section}

\newcommand\clf{\mathcal{F}}
\newcommand\cle{\mathcal{E}}
\newcommand\cln{\mathcal{N}}
\newcommand\clo{\mathcal{O}}
\newcommand\cll{\mathcal{L}}
\newcommand\clt{\mathcal{T}}
\newcommand\clp{\mathcal{P}}
\newcommand\cls{\mathcal{S}}
\newcommand\clsn{\mathcal{S}_n}
\newcommand\ra{\rightarrow}

\DeclareMathOperator{\cycr}{\mathcal{R}}
\DeclareMathOperator{\donwshift}{\mathcal{S}}

\newcommand{\rlab}[1]{\textcolor{red}{\scriptsize $#1$}}
\newcommand{\blab}[1]{\textcolor{blue}{\scriptsize $#1$}}
\newcommand{\lab}[1]{{\scriptsize $#1$}}

\DeclareMathOperator{\ewpmap}{\Psi}
\DeclareMathOperator{\newpmap}{\Psi}
\DeclareMathOperator{\ewlemap}{\Delta}
\DeclareMathOperator{\leewmap}{\Delta^{-1}}

\DeclareMathOperator{\newlemap}{\Gamma}

\DeclareMathOperator{\treelemap}{T}
\DeclareMathOperator{\letreemap}{\Lambda}

\DeclareMathOperator{\desexc}{\mathrm{DE}}
\DeclareMathOperator{\tr}{\mathrm{Tr}}

\definecolor{darkgreen}{rgb}{0,0.5,0}

\newcommand\ew{{\sc\small EW}}
\newcommand\new{{\sc\small NEW}}
\newcommand\pnew{{\sc\small (N)EW}}
\newcommand\ewt{{\sc\small EW}-tableau}
\newcommand\ewts{{\sc\small EW}-tableaux}
\newcommand\lt{{\sc\small Le}-tableau}
\newcommand\lts{{\sc\small Le}-tableaux}
\newcommand\ltx{{\sc\small Le}-tableaux}
\newcommand\lets{{\sc\small Le}-tableaux}
\newcommand\newt{{\sc\small NEW}-tableau}
\newcommand\newts{{\sc\small NEW}-tableaux}
\newcommand\tlt{tree-tableau}
\newcommand\tlts{tree-tableaux}

\newcommand{\az}{all-$0$}

\newcommand{\ao}{all-$1$}

\newcommand{\EWT}{\mbox{\sc{ewt}}}
\newcommand{\TT}{\mbox{\sc{tlt}}}
\newcommand{\MRC}{\mathrm{Rec}^m}

\begin{document}

\begin{frontmatter}


\title{EW-tableaux, Le-tableaux, tree-like tableaux and the Abelian sandpile model}

\author[add]{Thomas Selig\fnref{epsrc}}
\ead{thomas.selig@strath.ac.uk}
\ead[url]{https://personal.cis.strath.ac.uk/thomas.selig}

\author[add]{Jason P. Smith\fnref{epsrc2}}
\ead{jason.p.smith@strath.ac.uk}
\ead[url]{https://personal.cis.strath.ac.uk/jason.p.smith}

\author[add]{Einar Steingr\'imsson\fnref{epsrc}}
\ead{einar@alum.mit.edu}
\ead[url]{https://personal.cis.strath.ac.uk/einar.steingrimsson}

\fntext[epsrc]{This work was supported by grant EP/M015874/1 from The Engineering and Physical Sciences Research Council.}

\fntext[epsrc2]{This work was supported by grant EP/M027147/1 from The Engineering and Physical Sciences Research Council.}

\address[add]{Department of Computer and Information Sciences, University of Strathclyde, Glasgow G1 1XH, UK}

\begin{abstract}
  A EW-tableau is a certain 0/1-filling of a Ferrers diagram, corresponding uniquely to an acyclic orientation, with a unique sink, of a certain bipartite graph called a Ferrers graph.   We give a bijective proof of a result of Ehrenborg and van Willigenburg showing that EW-tableaux of a given shape are equinumerous with permutations with a given set of excedances.  This leads to an explicit bijection between EW-tableaux and the much studied Le-tableaux, as well as the tree-like tableaux introduced by Aval, Boussicault and Nadeau.

We show that the set of EW-tableaux on a given Ferrers diagram are in 1-1 correspondence with the minimal recurrent configurations of the Abelian sandpile model on the corresponding Ferrers graph.

 Another bijection between EW-tableaux and tree-like tableaux, via spanning trees on the corresponding Ferrers graphs, connects the tree-like tableaux to the minimal recurrent configurations of the Abelian sandpile model on these graphs.  We introduce a variation on the EW-tableaux, which we call NEW-tableaux, and present bijections from these to Le-tableaux and tree-like tableaux.  We also present results on various properties of and statistics on EW-tableaux and NEW-tableaux, as well as some open problems on these.
\end{abstract}


\end{frontmatter}

\section{Introduction and preliminaries}\label{sec-intro}

This paper presents a series of results on properties of the previously little studied \emph{\ewts} described by Ehrenborg and van Willigenburg~\cite{ew-ferrers}. Most of these results arise from a bijection we introduce between these tableaux and permutations.  As a consequence of that bijection we also give a bijection between \ewts\ and the much studied \lts, as well as to the \emph{tree-like tableaux} of Aval, Boussicault and Nadeau~\cite{aval-al-tree-tableaux}.  We refer to the latter simply as \emph{\tlts}.  Another
bijection that we present between \ewts\ and \tlts\ leads to a bijection between \tlts\ 
and the minimal recurrent configurations of the Abelian sandpile model on the Ferrers graphs described in~\cite{ew-ferrers}, via acyclic orientations with a unique sink on these bipartite graphs (as these orientations are in one-to-one
correspondence with the \ewts).  That bijection uses the bijection of Cori and Le Borgne~\cite{CLB} between the spanning trees of a graph and the recurrent configurations of the Abelian sandpile model on the graph.

We also present a variation on the \ewts, which we call {\newts}. In addition, we give a number of results on how properties of these two types of tableaux manifest themselves in properties of their corresponding permutations and in properties of orientations on the Ferrers graphs they are in a natural bijection with.

In \cite[Cor.~4.5]{ew-ferrers}, Ehrenborg and van Willigenburg described a set of 2-colorings of the cells of Ferrers diagrams satisfying certain conditions (a Ferrers diagram is a two-dimensional left-aligned array of boxes, weakly decreasing in the number of boxes when going down, see Figure~\ref{fig-first}).  These colorings  have interesting properties that the authors studied in connection with other structures related to such diagrams.  In particular, the set of colorings they define are in one-to-one correspondence with acyclic orientations, with a unique sink, of certain bipartite graphs, which they call Ferrers graphs.  We define here a set of 0/1-fillings of Ferrers diagrams that are essentially equivalent to the definition of Ehrenborg and van Willigenburg, with one minor difference that we explain in Remark~\ref{remark1} below.  We refer to the tableaux arising thus by the first letters of the last names of Ehrenborg and van Willigenburg.

\begin{definition}\label{def:ewt}  A \emph{\ewt} $\clt$ is a 0/1-filling of a Ferrers diagram with the following properties:
  \begin{enumerate}
  \item\label{ewt1} The top row of $\clt$ has a $1$ in every cell.
  \item\label{ewt2} Every other row has at least one cell containing a $0$.
  \item\label{ewt3}  No four cells of $\clt$ that form the corners of a rectangle have 0s in two diagonally opposite corners and 1s in the other two.
  \end{enumerate}
  The \emph{size} of a \ewt\ is one less than the sum of its number of rows and number of columns.
\end{definition}
An example of a \ewt\ is shown in Figure~\ref{fig-first}.

\begin{remark}\label{remark1}
Ehrenborg and van Willigenburg considered colorings corresponding to tableaux that in our definition would have a unique fixed, but arbitrary, row of all 1s, and no column with all 0s.  As all our results naturally restrict to tableaux of a fixed shape, it is irrelevant where this fixed all-1s row is chosen, and as it facilitates the presentation of our
results, we choose to always make this the top row.  That implies, of course, that no column has all 0s.  The choice of this row being arbitrary is directly linked to the fact, 
proved by Greene and Zaslavsky \cite[Thm. 7.3]{greene-zaslavsky-orientations}, that the number of acyclic orientations, with a unique sink, of a graph is independent of which vertex is chosen as the unique sink. In fact, it is easy to show that given an acyclic orientation with a unique sink, one can reverse the direction of every edge in every directed path from an arbitrary other vertex $v$ to the sink to obtain an acyclic orientation where $v$ is the unique sink.
\end{remark}

We label rows and columns on the southeast border of Ferrers diagrams underlying \ewts\ with the numbers $0,1,\ldots,n$, where the top row gets label $0$ and the successive border edges get the remaining numbers in order, as shown in Figure~\ref{fig-first}.

In \cite[Cor. 4.5]{ew-ferrers}, Ehrenborg and van Willigenburg proved the following theorem, in a different but equivalent form.  That these are equivalent follows from Lemma
2.2 in \cite{ehr-stein-excset}. The \emph{excedance bottoms set} of an $n$-permutation $a_1a_2\ldots a_n$ is the set of indices $i$ such that $a_i>i$. An $n$-permutation is a permutation of the integers $[n]=\{1,2,\ldots,n\}$.

\begin{theorem}\label{numtab}
  The \ewts\ of size $n$ with row labels $0,e_1,e_2,\ldots,e_k$ are equinumerous with $n$-permutations with excedance bottoms set $\{e_1,e_2,\ldots,e_k\}$.
\end{theorem}
The proof in \cite[Cor. 4.5]{ew-ferrers} was, however, not bijective.  One of the main results of the present paper is a bijection proving this, which leads to many new results on these tableaux, and their connections to other kinds of tableaux and other combinatorial objects. In particular, this leads to a bijection with the \lts, first introduced by Postnikov~\cite{postnikov-webs}, and then studied in a number of papers by various authors (see, for example, \cite{cort-mandel-will-koorn} and \cite{cor-nad} for
references). The \lts\ have been shown \cite{corteel-williams1, corteel-williams2} to encode the PASEP, a much studied one-dimensional lattice model in statistical mechanics, and to provide a refinement of statistics on the steady state distribution of that model.

\begin{definition}\label{def:let}
  A \emph{\lt} $\clt$ is a 0/1-filling of a Ferrers diagram, some of whose bottommost
 rows may be empty, satisfying the following properties:
  \begin{enumerate}
  \item Every column of $\clt$ has a 1 in some cell.
  \item If a cell has a 1 above it in the same column and a 1 to its left in the same row then it has a 1.
  \end{enumerate}
The \emph{size} of a \lt\ is the sum of its number of rows and number of columns.
\end{definition}
The difference in the definitions of size for \lts\ compared to \ewts\ is due to the fact that as defined both correspond to permutations whose length equals the size of the tableaux in question.

A bijection from \lts\ to permutations is essentially defined in \cite{postnikov-webs}, and explicitly described in \cite{stein-williams}, where various statistics preserved by this bijection are treated.  In particular, as was shown in \cite{postnikov-webs}, rows in a \lt\ correspond to \emph{weak excedances} in the corresponding permutation, that is, to indices $i$ in a permutation $a_1a_2\ldots a_n$ such that $a_i\ge i$.

A \emph{Ferrers graph} (see~\cite{ew-ferrers}) is a bipartite graph whose vertices of each part are labeled $t_0,t_1,\ldots,t_k$ and $b_1,b_2,\ldots,b_m$, respectively, satisfying the following conditions:
\begin{enumerate}
\item If $(t_i,b_j)$ is an edge and $r\le i$ and $s\le j$, then $(t_r,b_s)$ is also an edge.
\item Both $(t_0,b_m)$ and  $(t_k,b_1)$ are edges.
\end{enumerate}
Given a Ferrers diagram with rows labeled from top to bottom with $t_0,t_1,\ldots,t_k$ and
columns labeled with $b_1,b_2,\ldots,b_m$ from left to right,
there is a unique Ferrers graph whose vertices are labeled with the $t_i$ and $b_j$ and where $(t_i,b_j)$ is an edge if and only if the diagram has a cell in row $t_i$ and column $b_j$.  This correspondence is clearly one-to-one.

\begin{lemma}\label{4-cycle}
  If an orientation of a Ferrers graph contains a cycle, then it contains a 4-cycle.
\end{lemma}
\begin{proof}
Let $G$ be a Ferrers graph with top vertices $t_0,t_1,\ldots,t_k$ and bottom
vertices $b_1,b_2,\ldots,b_m$. Suppose an orientation $\clo$ of $G$
contains a cycle of length $2n$ for some $n>2$. Write this cycle as
$C=t_{i_1},b_{j_1},\ldots,t_{i_n},b_{j_n},t_{i_1}$ where the notation means
that edges are oriented $t_{i_q} \rightarrow b_{j_q}$ and $b_{j_q} \rightarrow t_{i_{q+1}}$, with the convention $i_{n+1} = i_1$.
Without loss of generality, we may assume that  $i_1 = \max \{ 0 \leq i \leq k 
; \, t_i \in C \}$.

Now since $\left( t_{i_1},b_{j_n} \right) $ is an edge of $G$ and $i_2 \leq     
i_1$ it follows that $\left(t_{i_2},b_{j_n}\right)$ is also an edge of     
$G$. There are two possibilities for its orientation in $\clo$.

\begin{enumerate}
\item If $\left(t_{i_2},b_{j_n}\right)$ is oriented $t_{i_2} \rightarrow b_{j_n}$, then                     
$t_{i_1},b_{j_1},t_{i_2},b_{j_n},t_{i_1}$ is a 4-cycle, and the proof is
complete.
\item If $\left(t_{i_2},b_{j_n}\right)$ is oriented $b_{j_n} \rightarrow t_{i_2}$,
 then  $C':=  t_{i_2},b_{j_2},\ldots,t_{i_n},b_{j_n},t_{i_2}$ is a cycle of length $2n-2$.
\end{enumerate}

In case 2, we simply iterate the reasoning on the cycle $C'$ until we reach a
cycle of length 4, as desired. 
\end{proof}

Since a cycle in an orientation of a Ferrers graph implies the existence of a 4-cycle, it follows that there is a one-to-one correspondence between acyclic orientations of a Ferrers graph, with a unique sink $t_0$, and \ewts\ on the corresponding Ferrers diagram, as pointed out in~\cite{ew-ferrers}.  Namely, an edge $(t_i,b_j)$ is oriented from $t_i$ to $b_j$ if and only if the cell in row $i$ and column $j$ has a 0. Thus, conditions \ref{ewt1} and \ref{ewt2} in Definition~\ref{def:ewt} correspond to $t_0$ being the unique sink, and condition \ref{ewt3} corresponds to there being no directed 4-cycles (and therefore no directed cycles at all).

Since we have a bijection from \ewts\ to permutations, and there exists a bijection between permutations and \lts, this induces a bijection between the two kinds of tableaux. However, the bijection we present here, to be described explicitly later, is constructed by first translating the permutation obtained from a \ewt, in two steps, into a different permutation.  The first translation is by the following bijection, which is a variation on the \emph{transformation fondamentale} of Foata and Sch\"utzenberger \cite{foata-schutz}.  This variation was described in the Appendix in \cite{es-indexed} (the author was not aware of this being a variation on the TF). The \emph{descent tops set} of a permutation $a_1a_2\ldots a_n$ is the set of those letters $a_i$ such that $a_i>a_{i+1}$ and the \emph{descent bottoms set} consists of those $a_{i+1}$ satisfying this.  The excedance bottoms set was defined before to be the set of indices $i$ such that $a_i>i$ and the \emph{excedance tops set} is the set of letters $a_i$ satisfying this.
\begin{proposition}\label{de-bij}
Let $\pi=a_1a_2\ldots a_n$ be an $n$-permutation.  The following defines a bijection\footnote{The name $\desexc$ refers to the map taking descent tops/bottoms to excedance tops/bottoms.} $\desexc:\clsn\ra\clsn$ such that $\desexc(\pi)=b_1b_2\ldots b_n$, where each $b_i$ is determined as follows:  Let $a_0=0$.  
\begin{enumerate}
\item If $a_i>a_j$ for some $j>i$ then $b_{a_{i+1}}=a_i$, that is, $a_i$ is then in place $a_{i+1}$ in $\desexc(\pi)$. 
\item If $a_i<a_j$ for all $j>i$, find the rightmost letter smaller than $a_i$.  If this letter is $a_k$ then $b_{a_{k+1}}=a_i$, that is, then $a_i$ is in place $a_{k+1}$ in $\desexc(\pi)$.
\end{enumerate}
The excedance bottoms set of  $\desexc(\pi)$ equals the descent bottoms set of $\pi$.
\end{proposition}
As an example, let $\pi=361542$.  Then $\desexc(\pi)=641523$: Since 3 has a smaller letter  to its right 3 goes to the place determined by the letter following~3, namely~6.  Since 1 has no smaller letter to its right we find the rightmost letter smaller than it, which is defined to be $a_0=0$, so 1 goes to place $a_1=3$, etc.

The second transformation, which we apply to the image of $\desexc$, is the cyclic right shift $\cycr$, that is, the bijection that maps  $a_1a_2\ldots a_n$ to $a_na_1a_2\ldots a_{n-1}$. The reason for this is that we can associate various properties of a \ewt~$\clt$ to properties of $\sigma=\desexc(\ewpmap(\clt))$, where $\ewpmap$ is the bijection taking \ewts\ to permutations, to be defined in the next section.  By right shifting $\sigma$ we turn excedances into weak excedances, in addition to 1 always becoming a weak excedance bottom when we move $a_n$, never an excedance top, to the first place, so we get exactly one more weak excedance in $\cycr(\pi)$ than we have excedances in $\pi$.

The paper is organized as follows:  In Section~\ref{sect-bij-ew-perm} we describe our bijection from \ewts\ to permutations.  In Section~\ref{sec-newt} we define \newts\ and describe some properties of the permutations obtained from those tableaux when applying to them the bijection from \ewts\ to permutations. We describe a bijection between \ewts\ and \lts\ in Section~\ref{sec:bij}, a bijection between \newts\ and \lts\ in Section~\ref{sect-bij-new-le} and a bijection between \ewts\ and \tlts\ in Section~\ref{sec-bij-new-tree}, the last of which can be easily modified to become a bijection between \newts\ and \tlts. In Section~\ref{sec:tableaux_trees_ASM} we present a connection between \tlts\ and certain spanning trees of the underlying Ferrers graph, and a connection between \ewts\ and minimal recurrent configurations of the Abelian sandpile model, and use this to give a new bijection between \tlts\ and \ewts, which is different from the bijection in Section~\ref{sec-bij-new-tree}.  Finally, in Section~\ref{sect-stats} we give several results on properties and statistics of \ew- and \newts\ and their corresponding permutations, and a few open problems on these. 

In a forthcoming paper \cite{dsss-rec-configs-asm}, we describe a certain ``decoration'' of the \ewts\ and their corresponding permutations.  There is a one-to-one correspondence between these decorated objects and all recurrent configurations of the Abelian sandpile model on Ferrers graphs. This complements and extends the work of Dukes and Le Borgne \cite{dukes-leborgne-polyominoes} (see also \cite{aval-dukes-al-operators}) on the sandpile model on the complete bipartite graph, which is the special case of Ferrers graphs corresponding to rectangular Ferrers diagrams.

\section{The bijection from \ewts\ to permutations}\label{sect-bij-ew-perm}

We now define a map $\ewpmap$ from \ewts\ of size $n$ to $n$-permutations, which we then show to be well defined, and subsequently that it is a bijection.  See Figure~\ref{fig-first} and Example~\ref{ex-ewt-perm} for an example of this.

\begin{definition}\label{phidef}
  Given a \ewt\ $\clt$, delete all entries from its top row and disregard the label 0 in what follows.  We construct a permutation $\pi=\ewpmap(\clt)$ by repeating the following two steps until all entries of~$\clt$ have been deleted:
\begin{enumerate}
\item\label{phi1} Write down labels of columns with no 1s, one after the other from right to left, and delete the entries from each of these columns.
\item\label{phi2} Write down labels of rows with no 0s, one after the other from bottom to top, and delete the entries from each of these rows.
\end{enumerate}
\end{definition}

\begin{figure}[h]
  \centering
  \begin{tikzpicture}[scale=0.35]
\begin{scope}[shift={(-2,0)}]
    \draw[step=2cm,thick] (0,2) grid (8,8);
    \draw[step=2cm,thick] (6,6) grid (10,8);
    \draw[step=2cm,thick] (0,0) grid (2,2);
    \node at (1,7) {$1$};
    \node at (3,7) {$1$};
    \node at (5,7) {$1$};
    \node at (7,7) {$1$};
    \node at (9,7) {$1$};
    \node at (1,5) {$0$};
    \node at (3,5) {$1$};
    \node at (5,5) {$0$};
    \node at (7,5) {$0$};
    \node at (1,3) {$0$};
    \node at (3,3) {$1$};
    \node at (5,3) {$0$};
    \node at (7,3) {$1$};
    \node at (1,1) {$0$};
   \node at (10.5,7) {\lab{0}};
    \node at (9,5.5) {\lab{1}};
    \node at (8.5,5) {\lab{2}};
    \node at (8.5,3) {\lab{3}};
    \node at (7,1.5) {\lab{4}};
    \node at (5.1,1.5) {\lab{5}};
    \node at (3.1,1.5) {\lab{6}};
    \node at (2.5,.7) {\lab{7}};
    \node at (1,-.5) {\lab{8}};
  \end{scope}    

\begin{scope}[shift={(11,0)}]
    \draw[step=2cm,thick] (0,2) grid (6,8);
    \draw[step=2cm,thick] (6,6) grid (8,8);
    \draw[thick] (0,0) -- (0,2);
    \node at (1,7) {$0$};
    \node at (3,7) {$1$};
    \node at (5,7) {$0$};
    \node at (7,7) {$1$};
    \node at (1,5) {$0$};
    \node at (3,5) {$0$};
    \node at (5,5) {$1$};
    \node at (1,3) {$1$};
    \node at (3,3) {$1$};
    \node at (5,3) {$1$};
    \node at (8.5,7) {\lab{1}};
    \node at (7,5.5) {\lab{2}};
    \node at (6.5,5) {\lab{3}};
    \node at (6.5,3) {\lab{4}};
    \node at (5,1.3) {\lab{5}};
    \node at (3,1.3) {\lab{6}};
    \node at (1,1.3) {\lab{7}};
    \node at (.5,.7) {\lab{8}};
    \node at (-.5,7) {\lab{1}};
    \node at (7,8.5) {\lab{2}};
    \node at (-.5,5) {\lab{3}};
    \node at (-.5,3) {\lab{4}};
    \node at (5,8.5) {\lab{5}};
    \node at (3,8.5) {\lab{6}};
    \node at (1,8.5) {\lab{7}};
    \node at (-.5,.7) {\lab{8}};
{\color{red}
\draw[thin] (-.3,7) -- (3,7)  -- (3,3)  -- (5,3)  -- (5,1.6) ;
    \draw[thin] (5,8.2)  -- (5,4.9)  -- (6.3,4.9) ; 
}
\end{scope}    
    
\begin{scope}[shift={(22,0)}]
    \draw[step=2cm,thick] (0,2) grid (6,8);
    \draw[step=2cm,thick] (6,6) grid (8,8);
    \draw[thick] (0,0) -- (0,2);
    \node at (1,7) {$1$};
    \node at (3,7) {$0$};
    \node at (5,7) {$0$};
    \node at (7,7) {$1$};
    \node at (1,5) {$1$};
    \node at (3,5) {$1$};
    \node at (5,5) {$0$};
    \node at (1,3) {$1$};
    \node at (3,3) {$1$};
    \node at (5,3) {$1$};
    \node at (8.5,7) {\lab{1}};
    \node at (7,5.5) {\lab{2}};
    \node at (6.5,5) {\lab{3}};
    \node at (6.5,3) {\lab{4}};
    \node at (5,1.5) {\lab{5}};
    \node at (3,1.5) {\lab{6}};
    \node at (1,1.5) {\lab{7}};
    \node at (.5,.7) {\lab{8}};
\end{scope}    
\node[] at (3, -2) {$15873426$};
\node[] at (15, -2) {$51473268$};
\node[] at (26, -2) {$84536127$};
\end{tikzpicture}
  \caption{Example of a \ewt\ (left), a \lt\ (center) and a \newt, with their corresponding permutations. Shown (in red) are two of the paths in the \lt\ determining its corresponding permutation (5 in the first place, 3 in the fifth, etc.), see \cite{stein-williams}. These paths are determined as follows:  From a label $\ell$ on the top or left border, head into the tableau.  When you hit a 1, turn right if you are heading down, down if you are heading right. The label you hit when exiting the tableau is the letter in place $\ell$ in the permutation. \label{fig-first}}
\end{figure}
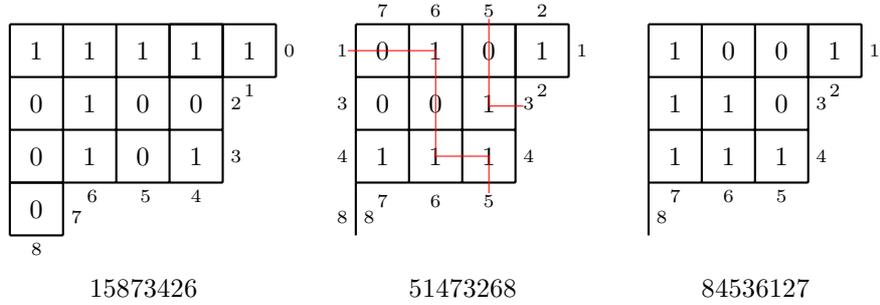

\begin{example}\label{ex-ewt-perm}
Figure~\ref{fig-first} shows an example of the bijection $\Psi$ in Definition~\ref{phidef} for the \ewt\ (on the left). Delete the entries in the top row. We first read the labels of columns with no~1s remaining, in increasing order: $1,5,8$; and delete the entries in these columns.  Then the labels of rows with no remaining~0s, in decreasing order, namely $7$ and $3$, deleting the entries in these rows, and then repeat.  The same algorithm is applied to the \newt\ (on the right), see Section~\ref{sec-newt}, except that we don't delete the top row.  In this case, however, there will be no columns without a 1, so we first read the rows 8 and 4, then column 5 and so on.
\end{example}

\begin{lemma}\label{well-def}
At each stage of the process in Definition \ref{phidef} there is an unread column with no 1s or an unread row with no 0s.
\end{lemma}
\begin{proof}
  If all columns have been read then no row has a 0, and if all rows have been read then no column has a 1. When there are both rows and columns unread, if each of those rows has a 0 and each of those columns has a 1, then we can go from a 0 in a row to a 1 in that 0's column, then to a 0 in that 1's row, ad infinitum, which would imply the corresponding
Ferrers graph had a directed cycle, contradicting Lemma~\ref{4-cycle} and its subsequent paragraph.
\end{proof}

Lemma \ref{well-def} shows that performing the process described in Definition \ref{phidef} we will read all the labels of rows and columns of $\clt$, and thus eventually produce an $n$-permutation.  The map $\ewpmap$ is therefore well defined.  We now show that it is a bijection, but we first need a lemma.

\begin{lemma}\label{desbot-lemma}
  When constructing the permutation $\pi=\ewpmap(\clt)$ of a \ewt\ $\clt$ by the algorithm described in Definition~\ref{phidef}, we first record a nonempty sequence of column labels in increasing order, then a nonempty sequence of row labels in decreasing order and so on, alternatingly.  Moreover, a row label is always preceded by a larger letter in $\pi$ and a column label is never preceded by a larger letter.
\end{lemma}
\begin{proof}
  After having deleted all entries from the top row of a \ewt\ $\clt$, since no other row can have all 1s and thus be 0-free, there must be a 1-free column, by Lemma \ref{well-def}.  On the first 
pass of part~\ref{phi1} in Definition \ref{phidef}, we read all the 1-free columns, in order of increasing labels and delete their entries. Having done that, there must now be a 0-free row, and we read the labels of all the 0-free rows, in decreasing order, according to \ref{phi2} in Definition \ref{phidef}, so that if a row label $\ell$ is read immediately after another row label $m$
we must have $m>\ell$.  

If a row label $\ell$ is immediately preceded by a column label $c$, then $c$ was the last column label read on that pass of column labels, and therefore the largest of those labels.  The label $\ell$ read just after that must belong to a row that had a 0 removed on the last pass of column labels read.  Therefore, $\ell$ cannot be larger than $c$, for then it would be larger than all the column labels read on the preceding pass, implying that the row in question wasn't long enough to have had a column with a label of $c$ or smaller.  Thus, a row label is always preceded by a larger letter in $\pi=\ewpmap(\clt)$.

An argument analogous to the above shows that, except for the first column label read, a column label is always preceded by a smaller letter in $\pi=\ewpmap(\clt)$, proving the claim.
\end{proof}

The following is an immediate consequence of Lemma \ref{desbot-lemma}.
\begin{proposition}\label{prop-rows-desbot}
  The labels of rows of a \ewt\ $\clt$ are exactly all the descent bottoms in $\pi=\ewpmap(\clt)$, and thus exactly all the excedance bottoms of~\,$\desexc(\pi)$.
\end{proposition}

\begin{proposition}\label{prop-bij}
The map $\ewpmap$ is a bijection.
\end{proposition}
\begin{proof}
  We show first that $\ewpmap$ is an injection.  If two tableaux $\clt$ and $\cls$ yield the same permutation then they must have the same shape, by Proposition~\ref{prop-rows-desbot}, and the same set of 1-free columns (after clearing the top row).  Once the entries of these columns have been deleted in both, $\clt$ and $\cls$ must have the same 0-free rows, since the labels of the same rows will be read on the next pass of the algorithm in Definition \ref{phidef}. But that implies that these rows were identical in $\clt$ and $\cls$, since both had 0s removed from all the same columns on the first pass.  By induction this must then apply to all columns and rows.

By Lemma~\ref{numtab}, applied to all possible sets of row labels, the set of all \ewts\ of size $n$ is equinumerous with $n$-permutations,
and thus injectivity of $\ewpmap$ implies that it is also surjective.
\end{proof}

It is straightforward, from the algorithm in Definition~\ref{phidef}, to verify the following.

\begin{proposition}\label{pro:perm_to_ewt}
  The following describes how to construct the inverse image~$\clt$ under $\ewpmap$ of a permutation $\pi=a_1a_2\ldots a_n$:

The set of descent bottoms of $\pi$, together with $n$,  determines the shape of $\clt$ and the labels of rows and columns. Fill the top row of $\clt$ with 1s. Now, since the letter $a_1$ is the first letter of the permutation and thus not a descent bottom, it is necessarily the label of a column; fill the empty cells of that column with~0s.  If $a_2$ labels a column, fill the empty cells of that column with 0s, otherwise fill the empty cells of the row labeled by $a_2$ with 1s.  Now repeat this for each~$a_i$.
\end{proposition}
This, in turn, obviously leads to the following description of the inverse of~$\ewpmap$:
\begin{corollary}\label{invbij}
  Let $\pi=a_1a_2\ldots a_n=\ewpmap(\clt)$.  If $r$ is the label of a row and $c$ the label of a column in $\clt$ then the cell in row $r$ and column $c$ has a 1 if and only if~$r$ precedes $c$ in~$\pi$.
\end{corollary}

An \emph{up-down run} in an $n$-permutation $\pi$ is a maximal sequence of consecutive letters in $\pi$ that is either increasing or decreasing, when we prepend 0 to~$\pi$ so that the first run is always increasing. For example, $\pi=51368427$ has up-down runs 05, 51, 1368, 842, 27.  The \emph{run-decomposition} of a permutation is its sequence of
up-down runs where the first letter of each run has been removed. Thus, $\pi$ has run-decomposition 5-1-368-42-7.  The number of up-down runs in~$\pi$ is thus one greater than the sum of the numbers of peaks and valleys in $0\pi$.  The distribution of the numbers of up-down runs in permutations is found in sequence A186370 in~\cite{oeis}.  
The following lemma is now straightforward to prove, the key observation being from Lemma~\ref{desbot-lemma} that the first row label read after reading a sequence of column labels must be smaller than the last column label read, and that the first column label read after reading row labels must be greater than the last row label read.

\begin{lemma}\label{run-decomp}
The blocks in the run-decomposition $R$ of a permutation $\pi=\ewpmap(\clt)$, for a \ewt\ $\clt$, correspond to the alternating sequences of column and row labels as they are read in constructing $\pi$ from $\clt$.
\end{lemma}

\section{\newts}\label{sec-newt}

If we reflect a \ewt\ along the NW-SE diagonal and turn 0s into 1s and vice versa, we get a tableau whose leftmost column is all 0. If we remove that leftmost column, but keep all its rows, the bottom few of which may now be empty, we get a tableau satisfying the conditions in the following definition.  Note that, by definition, every non-top row in a \ewt\ has a~0, which is reflected in condition~\ref{newt1} in the following definition.

\begin{definition}\label{def-new}
  A \emph{\newt}\footnote{{\sc\small{NEW}} stands for New \ewt.} $\clt$ is a 0/1-filling of a Ferrers diagram, possibly with some empty bottom rows, satisfying the following conditions:
  \begin{enumerate}
  \item\label{newt1} Every column of $\clt$ has a 1.
  \item\label{newt2}  No four cells of $\clt$ that form the corners of a rectangle have 0s in two diagonally opposite corners and 1s in the other two.
  \end{enumerate}
The \emph{size} of a \newt\ is the sum of its number of rows and number of columns.
\end{definition}

It is easy to see that Lemma~\ref{well-def} also applies to \newts, and therefore we can apply the map $\ewpmap$ to such a tableau $\clt$, yielding a permutation whose length equals the size of $\clt$.  The proof of Proposition~\ref{prop-bij} also goes through when~$\ewpmap$ is applied to a \newt, with the only change that we observe, for the last paragraph of the proof, that \newts\ of size $n$ are equinumerous with \ewts\ of size $n$, by the discussion preceding Definition~\ref{def-new}, and therefore with $n$-permutations.

The properties of a permutation $\pi=\ewpmap(\clt)$ for a \newt\ $\clt$ are slightly different from the case of \ewts.  Note that Lemma~\ref{well-def} ensures that some row in any \newt\ $\clt$ is all 1s, since every column of $\clt$ has a 1.

\begin{theorem}\label{props-new}
  Let $\pi=\ewpmap(\clt)$ for a \newt\ $\clt$.  Then $\pi$ and $\clt$ have the following properties:
  \begin{enumerate}
  \item The label of the lowest (possibly empty) 0-free
  row of $\clt$ equals the first letter in $\pi$, and thus is the place of the letter 1
   in $\desexc(\pi)$.
  \item The labels of other rows in $\clt$ are descent bottoms in $\pi$, and thus excedance bottoms of $\desexc(\pi)$.
\item If $r$ is the label of a row and $c$ the label of a column in $\clt$ then the cell in row $r$ and column $c$ has a 1 if and only if $r$ precedes $c$ in $\pi$.
  \end{enumerate}
\end{theorem}
\begin{proof}
  Since every column in $\clt$ has a 1, we will inevitably first read the label of a row when performing the algorithm described in Definition~\ref{phidef}.  Since row labels are read from bottom to top, we first read the label of the lowest row with no 0s (and the first letter in a permutation $\pi$ is always the place of 1 in $\desexc(\pi)$).  Apart from this, the argument in the proof of Lemma~\ref{desbot-lemma} applies here. 

Since the bijection $\ewpmap$ is defined in the same way for \newts\ as for \ewts, Proposition~\ref{pro:perm_to_ewt} and Corollary~\ref{invbij}
 apply here, except that we don't start by filling the top row with 1s, and the first letter of $\pi$ is always the label of a row.
\end{proof}


\section{A bijection between \ew- and \lts}\label{sec:bij}

Both \ew- and \lts\ of size $n$ map bijectively to $n$-permutations.  Moreover, the shape of a \ewt\ $\cle$ is determined by the excedance bottoms set of the permutation $\desexc(\ewpmap(\cle))$, and the shape of a \lt\ $\clt$ is determined by the weak excedance bottoms of the permutation $\Phi(\clt)$ described in~\cite{stein-williams} (see the caption of Figure~\ref{fig-first}).  
As pointed out before, by cyclically right shifting a permutation $\pi=a_1a_2\ldots a_n$ to obtain a permutation $\pi'=\cycr(\pi)$, the set of excedance tops in $\pi$ becomes the set of weak excedance tops in $\pi'$, except that $a_n$, which is never an excedance top goes to the first place in $\pi'$, adding one weak excedance top in $\pi'$. Thus, composing these bijections we get a bijection $\ewlemap$ from the set of \ewts\ of size $n$ to the set of \lts\ of size $n$, namely,
\begin{equation}\label{deltadef}
  \ewlemap = \Phi^{-1}\circ\cycr\circ\desexc\circ\ewpmap.
\end{equation}
Since all of the bijections in this composition have well understood inverses (see \cite{es-thesis} for the inverse of $\desexc$), it is clear how to define $\leewmap$.
We now describe how to effect the bijection $\leewmap$ directly from \lts\ to \ewts, and later give a direct description of $\ewlemap$.

Let $\cll$ be a \lt. We construct the corresponding \ewt\ $\cle=\leewmap(\cll)$.
The shape of $\cle$ is obtained by taking the shape of $\cll$ and adding a column to its left border, of the same length as that border, which may have empty rows at the bottom. We label the rows and columns of $\cll$ and $\cle$ in the following ways:

Starting at the top right corner of $\cll$, label the border edges increasingly, beginning with $1$, as you follow the path down to the bottom left. We call these labels the \emph{$r$-labels} of $\cll$. Repeat the same process for $\cle$ but begin with $0$. Label the left and top edges of $\cll$ by $e-1$, where $e$ is the $r$-label at the other end of that row or column, replace $0$ by $n$, and call these the \emph{$\ell$-labels} of $\cll$. See Figure~\ref{fig:bij} for an example of this labeling.  

The paths referred to in the description below are essentially the same paths as in \cite[Section~2]{stein-williams}, except that our paths here go in the ``backwards'' direction compared to \cite{stein-williams}.  Also, instead of having our paths go east or south on their last leg they end by going west or north, that is, our paths terminate at the opposite end of the row or column compared to the paths in \cite{stein-williams}, as shown in Figure~\ref{fig:bij}.  

\begin{figure}[h]
  \centering
  \begin{tikzpicture}[scale=0.45]
\begin{scope}[shift={(1,0)}]
    \draw[step=2cm,thick] (0,2) grid (6,8);
    \draw[step=2cm,thick] (6,6) grid (8,8);
    \draw[step=2cm,thick] (0,0) grid (4,2);
    \node at (1,7) {$0$};
    \node at (3,7) {$1$};
    \node at (5,7) {$1$};
    \node at (7,7) {$1$};
    \node at (1,5) {$1$};
    \node at (3,5) {$1$};
    \node at (5,5) {$1$};
    \node at (1,3) {$0$};
    \node at (3,3) {$0$};
    \node at (5,3) {$0$};
    \node at (1,1) {$0$};
    \node at (3,1) {$1$};
    \node at (8.5,7) {\rlab{1}};
    \node at (7,5.5) {\rlab{2}};
    \node at (6.5,5) {\rlab{3}};
    \node at (6.5,3) {\rlab{4}};
    \node at (5,1.5) {\rlab{5}};
    \node at (4.5,1) {\rlab{6}};
    \node at (3,-.5) {\rlab{7}};
    \node at (1,-.5) {\rlab{8}};
    \node at (1,8.5) {\blab{7}};
    \node at (3,8.5) {\blab{6}};
    \node at (5,8.5) {\blab{4}};
    \node at (7,8.5) {\blab{1}};
    \node at (-.5,7) {\blab{8}};
    \node at (-.5,5) {\blab{2}};
    \node at (-.5,3) {\blab{3}};
    \node at (-.5,1) {\blab{5}};

{\color{red}
  \draw[thin] (7,8.2)  -- (7,7.1)  -- (8.25,7.1); 
  \color{green}
  \draw[thin] (5,8.2)  -- (5,6.9)  -- (7.0,6.9) -- (7.0,5.8); 
  \color{blue}
  \draw[thin] (6.25,3) -- (-0.25,3) ;
}
\end{scope}    
    
\begin{scope}[shift={(14.5,0)}]
    \def\x{10}
    \draw[step=2cm,thick] (0,2) grid (8,8);
    \draw[step=2cm,thick] (6,6) grid (10,8);
    \draw[step=2cm,thick] (0,0) grid (6,2);
    \node at (1,7) {$1$};
    \node at (3,7) {$1$};
    \node at (5,7) {$1$};
    \node at (7,7) {$1$};
    \node at (9,7) {$1$};
    \node at (1,5) {$0$};
    \node at (3,5) {$0$};
    \node at (5,5) {$0$};
    \node at (7,5) {$0$};
    \node at (1,3) {$1$};
    \node at (3,3) {$1$};
    \node at (5,3) {$1$};
    \node at (7,3) {$0$};
    \node at (1,1) {$1$};
    \node at (3,1) {$0$};
    \node at (5,1) {$0$};
    \node at (10.5,7) {\rlab{0}};
    \node at (9,5.5) {\rlab{1}};
    \node at (8.5,5) {\rlab{2}};
    \node at (8.5,3) {\rlab{3}};
    \node at (7,1.5) {\rlab{4}};
    \node at (6.5,1) {\rlab{5}};
    \node at (5,-.5) {\rlab{6}};
    \node at (3,-.5) {\rlab{7}};
    \node at (1,-.5) {\rlab{8}};
  \end{scope}    
\node[] at (4, -2) {$51842736$};
\node[] at (11, -2) {$18427365$};
\node[] at (19.5, -2) {$14367582$};

  \end{tikzpicture}
  \caption{An example of a \lt\ (left) and corresponding \ewt\ (right) along with their $r$-labels (red) and $\ell$-labels (blue), used in Algorithm~\ref{leewalgo}.  Also shown are their corresponding permutations, and the intermediate permutation $\desexc(14367582)$, which is a cyclic left shift of 51842736. In the \lt, the first three paths considered by Algorithm~\ref{leewalgo}, starting at $r$-labels 1, 2 and 4, respectively, are shown in red, green and blue, respectively. \label{fig:bij}}
\end{figure}
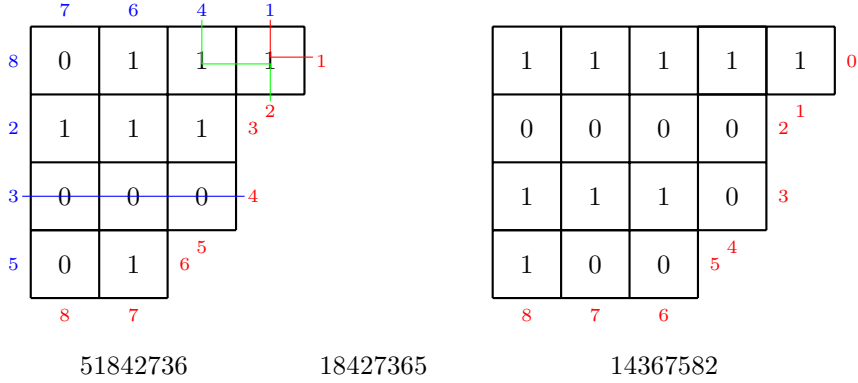

We now present the method for filling $\cle$ based on the filling of $\cll$. This also returns the permutation $\pi=\ewpmap(\cle)$.

\begin{algorithm}\label{leewalgo}
  \begin{enumerate}
\item Set all $r$-labels of $\cll$ as unvisited, set $\pi$ as the empty word and fill the top row of $\cle$ with $1$s.
\item\label{selecti} Let $i$ be the smallest unvisited $r$-label of $\cll$. If there are no unvisited $r$-labels left then stop.
\item\label{marki} Mark $i$ as visited and follow the path from the label $i$ until you hit an $\ell$-label, call it $j$. This is the path where when you hit a $1$ you turn west if you were going north and north if you were going west.
\item\label{step4} If $j$ is a column $r$-label of $\cle$ fill any empty cells of $\cle$ in that column with~$0$s. If $j$ is a row $r$-label of $\cle$ fill the empty cells in that row with~$1$s.
\item Add $j$ to the end of $\pi$. If $j$ is an unvisited $r$-label of $\cll$, then set $i=j$ and go to Step \ref{marki}, otherwise go to Step \ref{selecti}.
\end{enumerate}
\end{algorithm}
A tedious but straightforward argument, using the above algorithm for constructing $\cle=\leewmap(\cll)$, leads to the following description of how a \lt\ $\cll=\ewlemap(\cle)$ is constructed directly from the permutation $\pi=\ewpmap(\cle)$.  That in turn leads to the following direct construction of $\cll$ from $\cle$, since $\pi$ simply records the reading of columns and rows of $\cle$.

Let $m(i)$ be the position of the letter $i$ in $\desexc(\ewpmap(\cle))$.
Equivalently, $m(i)$ can be read directly from $\pi=\ewpmap(\cle)$, after we prepend 0 to $\pi$, as follows: Recall that the right-to-left minima of $\pi$ are the letters with no smaller letters to their right.  If the letter $i\ge1$ in $\pi$ is a right-to-left minimum let $m(i)$ be the letter immediately following the next right-to-left minimum to the left of $i$ in $\pi$, otherwise let $m(i)$ be the letter immediately to the right of $i$ in $\pi$.

The shape of $\cll$ is obtained by deleting the leftmost column of $\cle$, but keeping the possibly empty bottom rows resulting from that. We fill $\cll$ as follows, with~$i$ initialized to $1$, $n$ the size of $\cll$, and if $r=1$ we let $r-1=n$:

\begin{algorithm}\label{ewlealgo}
\begin{enumerate}
\item Enter $\cll$ from the $r$-label $i$, let $d\in\{N,W\}$ be the direction of travel and~$(r,c)$ the current cell, where $r$ and $c$ are the row and column~$r$-labels of that cell.
\item If $(r,c)$ is filled, then go to Step 4. If $(r,c)$ is empty, then fill it with $f(r,c)$ as follows:
\begin{itemize}
\item When $m(i)=r-1$, if $d=W$ then $f(r,c)=0$, else $f(r,c)=1$.
\item When $m(i)=c-1$, if $d=N$ then $f(r,c)=0$, else $f(r,c)=1$.
\item When $d=W$, $m(i)\ne r-1$  and there is a $1$ in the column above $(r,c)$, then $f(r,c)=1$.
\item When $d=W$, $m(i)\ne c-1$ and all cells in the column above $(r,c)$ are 0s, then $f(r,c)=0$.
\end{itemize}
\item If $(r,c)$ contains a $0$ move to the next cell in direction $d$ and if $(r,c)$ contains a $1$ move to the next cell in the ``opposite'' direction to $d$, that is, north if $d=W$ and west if $d=N$.
\item If we haven't reached the edge of $\cll$ go to Step 2. If we reach the edge of $\cll$, then increase $i$ by $1$, and if $i\le n$ go to Step 1, otherwise stop.
\end{enumerate}
\end{algorithm}

\begin{figure}[h]
  \centering
  \begin{tikzpicture}[scale=0.45]
  \begin{scope}[shift={(6,0)}]
    \draw[step=2cm,thick] (0,2) grid (8,8);
    \draw[step=2cm,thick] (6,6) grid (10,8);
    \draw[step=2cm,thick] (0,0) grid (6,2);
    \node at (1,7) {$1$};
    \node at (3,7) {$1$};
    \node at (5,7) {$1$};
    \node at (7,7) {$1$};
    \node at (9,7) {$1$};
    \node at (1,5) {$0$};
    \node at (3,5) {$0$};
    \node at (5,5) {$0$};
    \node at (7,5) {$0$};
    \node at (1,3) {$1$};
    \node at (3,3) {$1$};
    \node at (5,3) {$1$};
    \node at (7,3) {$0$};
    \node at (1,1) {$1$};
    \node at (3,1) {$0$};
    \node at (5,1) {$0$};
    \node at (10.5,7) {\rlab{0}};
    \node at (9,5.5) {\rlab{1}};
    \node at (8.5,5) {\rlab{2}};
    \node at (8.5,3) {\rlab{3}};
    \node at (7,1.5) {\rlab{4}};
    \node at (6.5,1) {\rlab{5}};
    \node at (5,-.5) {\rlab{6}};
    \node at (3,-.5) {\rlab{7}};
    \node at (1,-.5) {\rlab{8}};

\node[] at (4, -2) {$14367582$};
  \end{scope}  
  
 \node[] at (17, -2) {$18427365$};
  
\begin{scope}[shift={(20,0)}]
    \draw[step=2cm,thick] (0,2) grid (6,8);
    \draw[step=2cm,thick] (6,6) grid (8,8);
    \draw[step=2cm,thick] (0,0) grid (4,2);
    \node at (1,7) {};
    \node at (3,7) {$\textcolor{blue}{1}$};
    \node at (5,7) {$\textcolor{Green}{1}$};
    \node at (7,7) {$\textcolor{Brown}{1}$};
    \node at (1,5) {};
    \node at (3,5) {};
    \node at (5,5) {$\textcolor{blue}{1}$};
    \node at (1,3) {$\textcolor{YellowOrange}{0}$};
    \node at (3,3) {$\textcolor{YellowOrange}{0}$};
    \node at (5,3) {$\textcolor{YellowOrange}{0}$};
    \node at (1,1) {};
    \node at (3,1) {};
    \node at (8.5,7) {\textcolor{Brown}{\scriptsize $1$}};
    \node at (7,5.5) {\textcolor{Green}{\scriptsize $2$}};
    \node at (6.5,5) {\textcolor{blue}{\scriptsize $3$}};
    \node at (6.5,3) {\textcolor{YellowOrange}{\scriptsize $4$}};
    \node at (5,1.5) {\rlab{5}};
    \node at (4.5,1) {\rlab{6}};
    \node at (3,-.5) {\rlab{7}};
    \node at (1,-.5) {\rlab{8}};
    \draw[thin,Brown] (8.25,7) -- (7.25,7); \draw[thin,Brown] (7,7.4) -- (7,8);
    \draw[thin,Green] (7,5.75) -- (7,6.6);  \draw[thin,Green] (6.75,7) -- (5.25,7);\draw[thin,Green] (5,7.4) -- (5,8);
    \draw[thin,blue] (6.25,5) -- (5.25,5);  \draw[thin,blue] (5,5.4) -- (5,6.6);
    \draw[thin,blue] (4.75,7) -- (3.25,7);\draw[thin,blue] (3,7.4) -- (3,8);
    \draw[thin,YellowOrange] (6.25,3) -- (5.25,3);\draw[thin,YellowOrange] (4.75,3) -- (3.25,3);
    \draw[thin,YellowOrange] (2.75,3) -- (1.25,3);\draw[thin,YellowOrange] (0.75,3) -- (0,3);
\end{scope}    
  \end{tikzpicture}
  \caption{The \ewt\ $\clt$, and its associated permutations $\ewpmap(\clt)$ (left) and $\desexc(\ewpmap(\clt))$ (right), from Figure \ref{fig:bij} and the corresponding \lt\ partially filled by applying Algorithm \ref{ewlealgo}.\label{fig:bij2}}
\end{figure}

\section{A bijection between \new- and \lts}\label{sect-bij-new-le}

It is, of course, possible to define a bijection between \newts\ and \lts\ via permutations, in the exact same way as we did in the case of \ewts\ and \lts.  However, a variation on that turns out to have nicer properties, namely to preserve shape exactly.  As noted in Theorem~\ref{props-new}, the row labels of a \newt\ $\cln$ consist of the (number 
of the) place of~1 in the permutation $\pi=\desexc(\tau)$ and the excedance bottoms of $\pi$, where $\tau$ is the permutation read from $\cln$ according to Definition~\ref{phidef}. We call this set of excedance bottoms together with the place of 1 the \emph{augmented excedance bottoms set} of $\pi$.

If $n$ is the length of a permutation $\pi$ then the \emph{cyclic down-shift} $\donwshift$ subtracting~1 from each entry of $\pi$ and changing the resulting $0$ to $n$ takes the augmented excedance bottoms set of $\pi$ to the set of weak excedance bottoms of $\pi$. As
an example, $3164275\mapsto2753164$, where $\{1,2,3,6\}$ is the augmented excedance bottoms set of the first permutation and therefore the set of weak excedance bottoms of the second.

Thus, according to Theorem~\ref{props-new}, and the fact that the row labels of a \ewt\ $\clt$ are the excedance bottoms of the corresponding permutation $\desexc(\ewpmap(\clt))$,
we can define a bijection from \newts\ to \lts\ as the following composition (cf. Equation~\ref{deltadef}):
\begin{equation}
  \newlemap = \Phi^{-1}\circ\donwshift\circ\desexc\circ\ewpmap.
\end{equation}

With one modification the algorithm in~\ref{leewalgo} goes through to give a direct description of the above bijection.  Namely, instead of shifting the labels on the top and left border of a \lt\ $\cll$, we apply the transformation $\donwshift$ described above to the labels on the right of rows and bottom of columns.  We show an example of this map in Figure~\ref{fig-le-new}.

\begin{figure}[h]
  \centering
  \begin{tikzpicture}[scale=0.35]
\begin{scope}[shift={(-3,0)}]
    \draw[step=2cm,thick] (0,2) grid (6,8);
    \draw[step=2cm,thick] (6,6) grid (8,8);
    \draw[step=2cm,thick] (0,0) grid (4,2);
    \node at (1,7) {$0$};
    \node at (3,7) {$1$};
    \node at (5,7) {$1$};
    \node at (7,7) {$1$};
    \node at (1,5) {$1$};
    \node at (3,5) {$1$};
    \node at (5,5) {$1$};
    \node at (1,3) {$0$};
    \node at (3,3) {$0$};
    \node at (5,3) {$0$};
    \node at (1,1) {$0$};
    \node at (3,1) {$1$};
    \node at (8.5,7) {\rlab{2}};
    \node at (7,5.5) {\rlab{3}};
    \node at (6.5,5) {\rlab{4}};
    \node at (6.5,3) {\rlab{5}};
    \node at (5,1.5) {\rlab{6}};
    \node at (4.5,1) {\rlab{7}};
    \node at (3,-.5) {\rlab{8}};
    \node at (1,-.5) {\rlab{1}};

    \node at (1,8.5) {\blab{8}};
    \node at (3,8.5) {\blab{7}};
    \node at (5,8.5) {\blab{5}};
    \node at (7,8.5) {\blab{2}};
    \node at (-.5,7) {\blab{1}};
    \node at (-.5,5) {\blab{3}};
    \node at (-.5,3) {\blab{4}};
    \node at (-.5,1) {\blab{6}}; 
\end{scope}
   
\begin{scope}[shift={(6,0)}]
    \def\x{10}
    \draw[step=2cm,thick] (0+\x,2) grid (6+\x,8);
    \draw[step=2cm,thick] (6+\x,6) grid (8+\x,8);
    \draw[step=2cm,thick] (0+\x,0) grid (4+\x,2);
    \node at (1+\x,7) {$0$};
    \node at (3+\x,7) {$0$};
    \node at (5+\x,7) {$0$};
    \node at (7+\x,7) {$1$};

    \node at (1+\x,5) {$1$};
    \node at (3+\x,5) {$1$};
    \node at (5+\x,5) {$1$};

    \node at (1+\x,3) {$1$};
    \node at (3+\x,3) {$1$};
    \node at (5+\x,3) {$0$};

    \node at (1+\x,1) {$0$};
    \node at (3+\x,1) {$0$};

    \node at (8.5+\x,7) {\rlab{1}};
    \node at (7+\x,5.5) {\rlab{2}};
    \node at (6.5+\x,5) {\rlab{3}};
    \node at (6.5+\x,3) {\rlab{4}};
    \node at (5+\x,1.5) {\rlab{5}};
    \node at (4.5+\x,1) {\rlab{6}};
    \node at (3+\x,-.5) {\rlab{7}};
    \node at (1+\x,-.5) {\rlab{8}};
\end{scope}

\node[] at (0, -2) {$51842736$};
\node[] at (10, -2) {$62153847$};
\node[] at (20, -2) {$35478612$};

\end{tikzpicture}
\caption{An example of a \lt\ (left) and corresponding \newt\ (right) along with their $r$-labels (red) and $\ell$-labels (blue).  Also shown are their corresponding permutations, and the intermediate permutation $\desexc(35478612)$, which is a cyclic up-shift of 51842736. \label{fig-le-new}}
\end{figure}

\kern-4mm
\section{A bijection between \pnew- and  tree-tableaux}\label{sec-bij-new-tree}

In \cite{aval-al-tree-tableaux}, Aval, Boussicault and Nadeau introduced what they call \emph{tree-like tableaux}, Ferrers diagrams where some cells have a dot and others are empty.  We will here refer to these tableaux simply as \emph{tree-tableaux}.  

\begin{definition}\label{def-treetab}
  A \emph{tree-tableau} is a filling of a Ferrers diagram with the following properties:
\begin{enumerate}
\item The top left cell has a dot.
\item Every other cell has a dot above it (in the same column) or left of it (in the same row), but not both.
\item Every column and every row has a dot.
\end{enumerate}
The \emph{size} of a tree-tableau is one less than the sum of its number of rows and number of columns.
\end{definition}

Tree-tableaux have some very nice properties and strong connections to \ltx\ and the \emph{alternative tableaux} defined by Viennot~\cite{viennot-alt-tableaux}.  A bijection between tree-tableaux and \ltx\ can be inferred from  \cite[Prop.~3]{aval-al-tree-tableaux}.  We describe this bijection below, and how to translate it into a bijection to \ew- and \newts\ (see Figure~\ref{fig-tree-ew}) but first we introduce some notation. 

Define a dot in a tree-tableau as \emph{left-free} if it has no dots to its left in its row and \emph{up-free} if it has no dots above it in its column. By the definition of a tree-tableau every dot is either up-free or left-free, except the top left cell, which is both. Moreover, every row contains exactly one left-free dot and every column contains exactly one up-free dot. A cell has a dot \emph{weakly above} it if it  has a dot or there is a dot above it in the same column. Similarly a cell has a dot \emph{weakly right} of it if it has a dot or there is a dot to its right in the same row. A $0$ in a \lt\ is \emph{restricted} if it has a $1$ above it in the same column (see \cite{corteel-williams2}). 

The map $\treelemap$ from a tree-tableau $\clt$ to a \lt\ can be deduced from \cite[Prop.~3]{aval-al-tree-tableaux} and described as follows: Change each left-free dot to a $0$ and fill all cells to its left with $0$s. Change every up-free dot to a $1$ and fill all cells above it with $0$s. Finally, delete the leftmost column of $\clt$ and fill all remaining empty cells with $1$s.

The inverse map $\letreemap$, from a \lt\ $\cll$ to a tree-tableau, can be described as follows:  Add a column to the left of $\cll$ with a $1$ in the top cell and all $0$s below. 
Turn every $1$ that is highest in its column into a dot, turn the rightmost restricted $0$
in each row to a dot and then leave all remaining cells empty.

\def\fmap{\mathrm{F}}
We now define a map $\fmap$ from tree-tableau to \lt, which we show to be just a simpler description of $\treelemap$. Let $\clt$ be a tree-tableau and make the following changes:
\begin{itemize}
\item If a cell has a dot to its left and a dot weakly above, fill it with a $1$.
\item Otherwise fill the cell with a $0$.
\item Finally, delete the leftmost column to obtain $\fmap(\clt)$. 
\end{itemize}

\begin{lemma}
Given any tree-tableau $\clt$ we have $\treelemap(\clt)=\fmap(\clt)$.
\end{lemma}
\begin{proof}
Consider a cell $c$ in $\clt$ and let $\treelemap(c)$ and $\fmap(c)$ be the corresponding cells in $\treelemap(\clt)$ and $\fmap(\clt)$, respectively. We show that $\treelemap(c)$ and $\fmap(c)$ are filled with the same value for every cell in $\clt$.

Firstly suppose $c$ has no dot to the left, so $\fmap(c)$ contains a $0$. There is a left-free dot in cell $a$ that is weakly right of $c$, as every row has exactly one left-free dot.
Therefore, $\treelemap(a)$  and all cells to the left are filled with $0$s, so $\treelemap(c)$ contains~$0$. Next suppose $c$ has a dot to the left but no dot weakly above, so~$\fmap(c)$ contains a $0$. So there is an up-free dot (strictly) below $c$ in $\clt$, which implies $\treelemap(c)$ contains a $0$.

 Finally suppose $c$ has a dot to the left 
 and a dot weakly above, 
 so $\fmap(c)$ contains a $1$. If $c$ contains a dot, then the dot must be up-free because it is not left-free as there is a cell to the left with a dot, which implies~$\treelemap(c)$ contains a $1$. If $c$ does not contain a dot then the left-free dot of that row must be left of $c$, and the up-free dot of that column must be above $c$, which implies~$\treelemap(c)$ is one of the remaining empty cells that get filled with a~$1$.
\end{proof}

In Section \ref{sec:bij} we defined a bijection $\leewmap$ from \lets- to \ewts. Composing this bijection with the bijection between tree-tableaux and \lets\ gives a bijection from tree-tableaux to \ewts. The map from a \lt\ $\mathcal{L}$ to a \ewt\ $\cle$ is given by Algorithm~\ref{leewalgo} by tracing paths through $\cll$, from an $r$-label to an $\ell$-label. We can define a path through a tree-tableau which changes direction in a cell $c$ if $\fmap(c)$ contains a $1$. So the paths through the tree-tableau are identical to the paths through the corresponding \lt, apart from an additional horizontal step through the cell in the extra leftmost column, regardless of its content.  A path through a tree-tableau thus changes direction in a cell $c$ according to the following rule:
\begin{equation}\label{eq:rule-tree-ew}
  \text{\emph{Change direction if and only if $c$ has a dot to the left and a dot weakly above.}}
\end{equation}

The path through a tree-tableau beginning at $r$-label $i$ thus ends at the same $\ell$-label as the path through the corresponding \lt.  
We can therefore apply Algorithm~\ref{leewalgo} to trace these paths through a tree-tableau to get a map from tree-tableaux to \ewts, see Figure~\ref{fig-tree-ew}. By the last paragraph of Section~\ref{sect-bij-new-le}, we can thus also use Algorithm~\ref{leewalgo} to produce a bijection from tree-tableaux to \newts.

\begin{figure}[h]
  \centering
  \begin{tikzpicture}[scale=0.3]
\begin{scope}[shift={(-3,0)}]
    \draw[step=2cm,thick] (0,0) grid (4,10);
    \draw[step=2cm,thick] (4,4) grid (8,10);
    \draw[step=2cm,thick] (8,6) grid (10,10);
    \node at (10.5,9) {\rlab{1}};
    \node at (10.5,7) {\rlab{2}};
    \node at (9,5.5) {\rlab{3}};
    \node at (8.5,5) {\rlab{4}};
    \node at (7,3.5) {\rlab{5}};
    \node at (5,3.5) {\rlab{6}};
    \node at (4.5,3) {\rlab{7}};
    \node at (4.5,1) {\rlab{8}};
    \node at (3,-.5) {\rlab{9}};

    \node at (3,10.5) {\blab{8}};
    \node at (5,10.5) {\blab{5}};
    \node at (7,10.5) {\blab{4}};
    \node at (9,10.5) {\blab{2}};
    \node at (-.5,1) {\blab{7}};
    \node at (-.5,3) {\blab{6}};
    \node at (-.5,5) {\blab{3}};
    \node at (-.5,7) {\blab{1}};
    \node at (-.5,9) {\blab{9}};
    
    \draw [fill=black] (1,1) circle [radius=0.25];
    \draw [fill=black] (3,3) circle [radius=0.25];
    \draw [fill=black] (7,5) circle [radius=0.25];
    \draw [fill=black] (1,7) circle [radius=0.25];
    \draw [fill=black] (3,7) circle [radius=0.25];
    \draw [fill=black] (5,7) circle [radius=0.25];
    \draw [fill=black] (1,9) circle [radius=0.25];
    \draw [fill=black] (7,9) circle [radius=0.25];
    \draw [fill=black] (9,9) circle [radius=0.25];
    
    \draw[color=red,thin] (9,10.2)  -- (9,9.1)  -- (10.25,9.1); 
    \draw[color=green,thin] (7,10.2)  -- (7,8.9)  -- (9,8.9) -- (9,6.9) -- (10.2,6.9); 
    \draw[color=blue,thin] (8.25,5) -- (-0.25,5) ;

\end{scope}
   
\begin{scope}[shift={(9,0)}]
    \draw[step=2cm,thick] (2,0) grid (4,10);
    \draw[step=2cm,thick] (4,4) grid (8,10);
    \draw[step=2cm,thick] (8,6) grid (10,10);
    \node at (10.5,9) {\rlab{1}};
    \node at (10.5,7) {\rlab{2}};
    \node at (9,5.5) {\rlab{3}};
    \node at (8.5,5) {\rlab{4}};
    \node at (7,3.5) {\rlab{5}};
    \node at (5,3.5) {\rlab{6}};
    \node at (4.5,3) {\rlab{7}};
    \node at (4.5,1) {\rlab{8}};
    \node at (3,-.5) {\rlab{9}};

    \node at (3,10.5) {\blab{8}};
    \node at (5,10.5) {\blab{5}};
    \node at (7,10.5) {\blab{4}};
    \node at (9,10.5) {\blab{2}};
    \node at (1.5,1) {\blab{7}};
    \node at (1.5,3) {\blab{6}};
    \node at (1.5,5) {\blab{3}};
    \node at (1.5,7) {\blab{1}};
    \node at (1.5,9) {\blab{9}};
    
    \node at (3,1) {$1$};
    \node at (3,3) {$0$};
    \node at (3,5) {$0$};
    \node at (3,7) {$1$};
    \node at (3,9) {$0$};
    \node at (5,5) {$0$};
    \node at (5,7) {$1$};
    \node at (5,9) {$0$};
    \node at (7,5) {$0$};
    \node at (7,7) {$1$};
    \node at (7,9) {$1$};
    \node at (9,7) {$1$};
    \node at (9,9) {$1$};
    
\end{scope}

\begin{scope}[shift={(22,0)}]
    \draw[step=2cm,thick] (0,0) grid (4,10);
    \draw[step=2cm,thick] (4,4) grid (8,10);
    \draw[step=2cm,thick] (8,6) grid (10,10);
    
    \node at (10.5,9) {\rlab{0}};
    \node at (10.5,7) {\rlab{1}};
    \node at (9,5.5) {\rlab{2}};
    \node at (8.5,5) {\rlab{3}};
    \node at (7,3.5) {\rlab{4}};
    \node at (5,3.5) {\rlab{5}};
    \node at (4.5,3) {\rlab{6}};
    \node at (4.5,1) {\rlab{7}};
    \node at (3,-.5) {\rlab{8}};
    \node at (1,-.5) {\rlab{9}};
    
    \node at (1,1) {$0$};
    \node at (1,3) {$0$};
    \node at (1,5) {$1$};
    \node at (1,7) {$0$};
    \node at (1,9) {$1$};
    \node at (3,1) {$1$};
    \node at (3,3) {$1$};
    \node at (3,5) {$1$};
    \node at (3,7) {$0$};
    \node at (3,9) {$1$};
    \node at (5,5) {$1$};
    \node at (5,7) {$1$};
    \node at (5,9) {$1$};
    \node at (7,5) {$0$};
    \node at (7,7) {$0$};
    \node at (7,9) {$1$};
    \node at (9,7) {$0$};
    \node at (9,9) {$1$};
\end{scope}

  \end{tikzpicture}
\caption{An example of a \tlt\ (left) and corresponding \lt\ (center) and \ewt\ (right). In the \tlt, the first three paths given by rule~\eqref{eq:rule-tree-ew}, starting at $r$-labels 1, 2 and 4, are shown in red, green and blue, respectively.\label{fig-tree-ew}}
\end{figure}
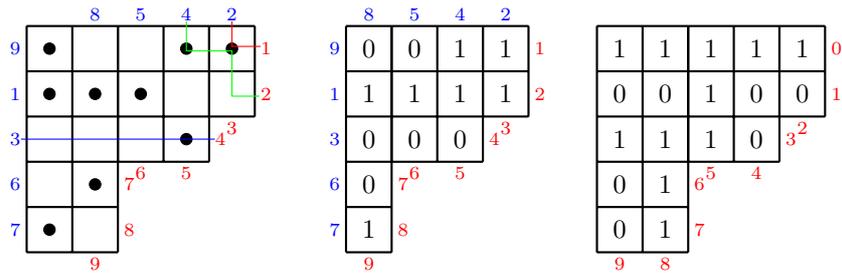


\section{Tableaux, spanning trees and the Abelian sandpile model}\label{sec:tableaux_trees_ASM}

In this section we exhibit a connection between the \tlts\ and certain spanning trees of the underlying Ferrers graph. We combine this with the work in \cite{CLB} to obtain a new bijection between \tlts\ and \ewts. This bijection is different from that of Section~\ref{sec-bij-new-tree}. 

\subsection{\ewts\ and minimal recurrent configurations of the Abelian sandpile model}\label{subsec:EWT_ASM}

The Abelian sandpile model (ASM) is a dynamic process on
a graph $G=(V \cup \{s \},E)$ with a special distinguished vertex $s$ called the sink. More precisely, it is a Markov chain on the set of \emph{configurations} on $G$.
We consider the ASM on Ferrers graphs. Let $G(F)$ be a Ferrers graph, with vertex set $0,1,\ldots,n$ where we identify a vertex with the label of its row or column in the corresponding Ferrers diagram. We will always take vertex $0$, the label of the top row of the diagram, to be the sink of the ASM.

A \emph{configuration on $G(F)$} is a vector $c=(c_1,\ldots,c_n)$ of nonnegative integers. 
We think of $c_i$ as the number of ``grains of sand'' at the vertex $i$.
Given a configuration $c$, we say that the vertex $i$ is \emph{stable} if $c_i < d_i$, where $d_i$ is the degree of vertex $i$, that is, the number of cells in the row or column labeled $i$ in the Ferrers diagram. A configuration is stable if all its vertices are stable.

Unstable vertices \emph{topple}, sending one grain to each of their neighbors. This may cause some of these neighbors to become unstable themselves, and topple in turn. The sink never topples, and can be viewed as absorbing grains.
It is possible to show (see for instance \cite[Section 5.2]{Dhar}) that starting from any configuration $c$ and toppling unstable vertices, one eventually reaches a stable configuration $c'$. 
Moreover, $c'$ does not depend on the order in which vertices are toppled in this sequence. We call $c'$ the \emph{stabilization} of $c$ and write $c'=\sigma(c)$.

For a configuration $c$, define $\tilde{c}$ by 
$$
\tilde{c}_i := 
\begin{cases} 
c_i + 1, \quad \text{if $i$ is a column label},\\ c_i, \quad \text{otherwise}. 
\end{cases}
$$ 
A configuration $c$ is \emph{recurrent} if $\sigma(\tilde{c}) = c$. Moreover, in the stabilization of $\tilde{c}$, every vertex $i$ for $1 \leq i \leq n$ topples exactly once \cite[Section 6.1]{Dhar}. 
A recurrent configuration $c$ is \emph{minimal} if it is minimal with respect to the total number of grains $\sum_{i=1}^n c_i$. The set of minimal recurrent configurations of a graph $G$ is denoted $\MRC(G)$.

Let $F$ be a Ferrers diagram. To a \ewt\ $\cle$ of shape $F$, we associate a configuration $c=c(\cle)$ on the Ferrers graph $G=G(F)$ as follows:
\begin{itemize}
\item If $i$ is the label of a column $C$, then $c_i$ is the number of $0$s in $C$.
\item If $i$ is the label of a row $R$, then $c_i$ is the number of $1$s in $R$.
\end{itemize}
Since each row of a \ewt\ contains at least one $0$, and each column at least one $1$ (in the top row), $c$ is stable by construction. Combining the results from \cite[Cor. 3]{Sch} and Section \ref{sec-intro} on acyclic orientations, we get the following result, where $\EWT(F)$ is the set of \ewts\ of shape $F$.
\begin{theorem}\label{thm:EWT_ASM}
The map $\EWT(F)\longrightarrow\MRC(G)$ that maps $\cle$ to $c=c(\cle)$
is a bijection. Moreover, let $\tau=\Phi(\cle)$ be the permutation obtained by Definition~\ref{phidef}. Then, starting from the configuration $\tilde{c}$, the vertices of $G$ can be toppled in the order of $\tau$, read from left to right, to reach the configuration $c$.
\end{theorem}

\begin{proof}
Let $G=G(F)$ be a Ferrers graph. Write $O(G)$ for the set of acyclic orientations of $G$ where the vertex labeled $0$ corresponding to the top row of $F$ is the unique sink. 
For an orientation $\clo\in O(G)$, we define a configuration $c=c(\clo)$ on $G$ by 
$c_i=\mathrm{in}_i$, the number of incoming edges to the vertex $i$ in $\clo$. 
In terms of the ASM, and after switching the orientations of all edges, Corollary 3 in \cite{Sch} can be restated as follows: The map $O(G) \longrightarrow \MRC(G)$ that maps $\clo$ to $c=c(\clo)$ is a bijection. Combined with the observation of a one-to-one correspondence between $\EWT(F)$ and $O(G(F))$ in Section~\ref{sec-intro}, this yields the desired bijection.

The toppling order of $\tilde{c}$ is the same as the \emph{firing sequence} in \cite{Sch}. For an orientation $\clo$, this is given by removing the sink and
all incoming edges, recording all vertices that are now sinks, removing these, and iterating. This is the exact analogue of the construction of the bijection $\Psi$ in Definition~\ref{phidef}, though in the toppling order the order in which we record vertices in
each block in the run-decomposition of $\tau$ is unimportant.
\end{proof}

\subsection{Tree-tableaux and spanning trees}\label{subsec:TLT_ST}

Let $F$ be a Ferrers diagram, and $G=G(F)$ the corresponding Ferrers graph. A \emph{dotted tableau} is a filling of $F$ where each cell is either empty or contains a single dot. A \tlt\ is thus an example of a dotted tableau. To any dotted tableau $ \clt $ we may associate a subgraph $S=S(\clt)$ of $G$ by keeping the edges corresponding to dotted cells, and removing those corresponding to empty cells. This clearly gives a bijective correspondence between dotted tableaux and subgraphs of $G$. Now let $<_E$ be a total order on the
edges of $G$, and $S$ be a spanning tree of $G$. An edge $e \notin S$ is \emph{externally active} if it is the minimal edge in the unique cycle of $S \cup \{e\}$ for the order $<_E$. The external activity $e(S)$ of a spanning tree $S$ is its number of externally active edges. In this subsection we identify a cell in a Ferrers diagram with its corresponding edge in the Ferrers graph.

\begin{definition}
Let $G=G(F)$ be a Ferrers graph, and $<_E$ a total order on its edges or, equivalently, on the cells of $F$. We say that $<_E$ is \emph{$F$-compatible} if it is increasing in every row from left to right, and every column from top to bottom, in $F$.
\end{definition}

\begin{theorem}\label{thm:tree-tt}
Let $G=G(F)$ be a Ferrers graph, and $<_E$ an $F$-compatible order. Let $ \clt $ be a dotted tableau of shape $F$. Then $ \clt $ is a \tlt\ if and only if $S(\clt)$ is a spanning tree with external activity $0$.
\end{theorem}

\begin{proof}
First suppose that $ \clt $ is a \tlt , and let $S=S(\clt)$. We first show that $S$ contains no cycles. Assume that $S$ does contain a cycle $C$, and consider the lowest row in $ \clt $ that contains a dotted cell belonging to $C$, and take $c$ to be the rightmost dotted cell in that row belonging to $C$. Then each vertex of $c$ must be connected by an edge to at 
least one other vertex in the cycle $C$, and by construction this means that $c$ must have a dotted cell above it in its column, and to its left in its row, which contradicts the definition of a \tlt . Thus $S$ contains no cycles. Moreover, from~\cite{aval-al-tree-tableaux}, we have that the number of dotted cells in $ \clt $, that is, of edges in $S$, is one less than the number of vertices of $G$, and thus $S$ is a spanning tree.

We now need to show that $e(S)=0$. First we will prove two useful lemmas. Since $S$ is a 
spanning tree, for any pair of vertices $v,w$ in the graph $G$, there is a unique path from $v$ to $w$, which we denote $v \rightarrow w$.
\begin{lemma}\label{lem:zigzag}
Let $v$ be a vertex of $S$. Then on the unique 
path $v \rightarrow 0$ in $S$, the row labels of vertices appear in decreasing order, and the column labels in increasing order.
\end{lemma}

\begin{proof}
Given a \tlt\ $ \clt $ and a vertex $v$ of $S(\clt)$ we obtain the path from $v \rightarrow 0$ via the following steps,
\begin{itemize}
\item if $v$ is a row of the Ferrers diagram, set $v'$ to be the leftmost column of $ \clt $ that contains a dot in the row $v$;
\item if $v$ is a column of the Ferrers diagram, set $v'$ to be the topmost row of $ \clt $ that contains a dot in the column $v$;
\end{itemize}
and iterating the above on $v'$ until reaching the top row, which is the vertex~$0$. Since going from bottom to top decreases row labels, and going from right to left increases column labels, we have the desired result.
\end{proof}

\begin{lemma}\label{lem:order_ineq}
Let $(i,j)$ and $(i',j')$ be two edges of $G$ such that $i' \leq i$ and $j' \geq j$, that is, the cell $(i',j')$ is above and to the left of the cell $(i,j)$. Then $(i',j') \leq_E (i,j)$.
\end{lemma}

\begin{proof}
Since $F$ is a Ferrers diagram, the edge $(i',j)$ is in $G$, and the corresponding cell is to the right of $(i',j')$ in the same row and above $(i,j)$ in the same column. The order $<_E$ being $F$-compatible, this implies $(i',j') \leq_E (i',j) \leq_E (i,j)$ as desired.
\end{proof}

We return to the proof that $e(S)=0$. Let $e=(i,j)$ be an edge that isn't in $S$, where $i$, resp. $j$, is the label of the row, resp. column, of the cell $e$. The unique cycle $C$ in $S \cup \{e\}$ is obtained by taking the union of the paths $i \rightarrow 0$ and $j \rightarrow 0$, removing edges that appear in both paths, and combining with the edge $e$. Let $i'$ be the first row label encountered on the path $j \rightarrow 0$ such that $i' \leq i$, and $j'$ the column label of the vertex immediately preceding $i'$ on this path. 
By Lemma~\ref{lem:zigzag}, $j'\geq j$. If $(i',j')$ isn't on the path $i \rightarrow 0$, then it belongs to $C$, and by Lemma~\ref{lem:order_ineq}, we have $(i',j') <_E (i,j)$. The inequality is strict since $(i,j)$ is not on the path $j \rightarrow 0$ (it isn't in $S$). Thus $(i,j)$ is not externally active, since $C$ contains an edge that is strictly smaller in $<_E$. If the path $i \rightarrow 0$ does go through $j'$ and $i'$ in that order (it must go through them in the same order as the path $j \rightarrow 0$), then let $i''$ be the vertex immediately preceding $j'$ on the path $i \rightarrow 0$. As above, by Lemma~\ref{lem:zigzag}, $i'' \leq i$. Since $i'$ is the first vertex encountered on the path $j \rightarrow 0$ with $i' \leq i$, and by construction this path encounters $i''$ before $i'$, this implies that $(i'',j')$ is not in the path $j \rightarrow 0$. Moreover, from Lemma~\ref{lem:order_ineq} we get that $(i'',j') <_E (i,j)$, which shows that $(i,j)$ is not externally active. We have thus shown that if $ \clt $ is a \tlt , then $S=S(\clt)$ is a spanning tree of $G$ with $e(S)=0$.

For the converse, suppose that $S=S(\clt)$ is a spanning tree with $e(S)=0$. We show that $ \clt $ is a \tlt . Since $<_E$ is $F$-compatible, it follows that the edge $e^{\rho}$ corresponding to the top left corner of $F$ is minimal in $<_E$. Therefore it must be in the spanning tree $S$, since otherwise it would be externally active. Thus the top left corner of $ \clt $ is dotted. Now consider a dotted cell $e_1$ of $ \clt $, and suppose it has both a dot above it in its column, say some $e_2$, and a dotted cell to its left in its row, say some $e_3$. Consider the cell $e_4$ that is in the same row as $e_2$ and column as $e_3$. Since $e_1,e_2,e_4,e_3$ is a cycle of $F$, and $S$ is a spanning tree, $e_4$ cannot be in $S$. But the condition that $<_E$ is $F$-compatible implies that $e_4 <_E e_2 <_E e_1$ and $e_4 <_E e_3 <_E e_1$, and thus $e_4$ is externally active, hence the desired contradiction.

It remains to show that every dotted cell $e$ has a dotted cell either above it in its column or to its left in its row. Let $e=(i,j)$ where $i$ is the label of its row and $j$ of its column. Since $S$ is a spanning tree, there are paths $i \rightarrow 0$ and $j \rightarrow 0$ using edges of $S$, and exactly one of these paths starts along the edge~$e$. Suppose this is the path $i \rightarrow 0$. We write this path as $i,e,j,e',i',\ldots,0$. It is sufficient to show that the labels $i,i'$ satisfy $i'<i$ since then the cell $e'=(i',j)$ is a dotted cell above $e$ in its column. Seeking contradiction, suppose that $i'>i$. Let $e''=(i'',j'')$ be the first edge on the path $i \rightarrow 0$ such that $i''<i$ or $j''<j$. By construction exactly one of these conditions will hold. If we have $j''<j$,  then the edge $(i,j'')$ is externally active, since it is minimal in the cycle $i,j,i',\ldots,i'',j'',i$. Similarly, if $i''<i$ then the edge $(i'',j)$ is externally active. This is the desired contradiction, and thus $i'<i$ as desired. We can show in exactly the same fashion that if it is the path $j \rightarrow 0$ that goes through $e$, then there is a dotted cell to the left of $e$ in its row.
\end{proof}

Thus, fixing an order $<_E$ that is $F$-compatible, we have a bijection between \tlts\ and spanning trees with external activity equal to $0$. We now use the work of Cori and Le Borgne \cite{CLB} to build a new bijection from \tlts\ to \ewts.

\subsection{Tree-tableaux and \ewts\ }\label{sec:tree-ew}

Let $G=G(F)$ be a Ferrers graph on $n+1$ vertices,
labeled $0,\ldots,n$.  We present an algorithm that constructs an $n$-permutation $\tau=v_1\ldots v_n$ from a \tlt~$\clt$ on the Ferrers diagram corresponding to $G$. This allows us to construct a bijection between \tlts\ and \ewts, which is different from that of Section~\ref{sec-bij-new-tree}. 
\begin{algorithm}\label{algo-tree-tab-perm}
Given a \tlt\ $\clt$ with row and column labels $0,\ldots,n$, initialize $v_0=0$ and~$i=1$, and determine each $v_i$ for $i>0$ as follows:
\begin{enumerate}
\item\label{step1} Set $j=i-1$.
\item\label{step2} If $v_j$ is a column (resp. row) let $e_i$ be the bottommost (resp. rightmost) dotted cell in $v_j$ not in $\{e_1,\ldots,e_{i-1}\}$. If no such cell exists decrease $j$ by one and repeat Step~\ref{step2}.
\item Let $v_i$ be the row or column of $e_i$ not in $\{v_0,\dots,v_{i-1}\}$.
\item Increase $i$ by one, if $i=n+1$ then stop, otherwise go to Step~\ref{step1}.
\end{enumerate}
\end{algorithm}

The algorithm above can be described in the following way: 
Let $e_1$ be the rightmost dotted cell in the top row, set $v_1$ as the column of $e_1$, and say that this cell and column have been visited. We then take the bottommost dotted cell in column $v_1$, we set this as $e_2$ and $v_2$ as the row of $e_2$. We then repeat this with the rightmost dotted cell in $v_2$. We continue this until there are no unvisited dotted cells in the current row/column $v_i$, at which point
we go backwards through the visited rows and columns $v_{i-1},\ldots,v_1$ until we find the first row/column that does have an unvisited dotted cell.  We then take the rightmost/bottommost such cell, and then repeat the previously described process until all rows and columns have been visited.

\begin{example}\label{ex:TLT_perm1}
Let $ \clt $ be the \tlt\ in Figure~\ref{fig:ex_tlt_perm1}, with rows and columns labeled. We construct the permutation $\tau=v_1\ldots v_n$ by Algorithm~\ref{algo-tree-tab-perm}. Start with $v_0 = 0$, which is the top row, and set $i=1$ and $j=0$. Since $v_j=0$ is a row and $i=1$, we set $e_1$ to be the rightmost dotted cell in $v_0$, that is, $ e_1 = (0,3)$. We then set $v_1=3$, increase $i$ by one, so $i=2$, and return to Step 1, setting $j=1$. Now $v_j = v_1 = 3$, which is a 
column, so we let $e_2$ be the bottommost dotted cell in the column $3$ that is not in $\{e_1\} = \{ (0,3) \}$. Thus $e_2 = (1,3)$, $v_2 = 1$, and we set $i=3$ and return to Step 1. Now we set $j=2$, so $v_j = v_2 = 1$, which is a row. Thus $e_3$ is the rightmost dotted cell in the row $1$ that isn't in $\{e_1,e_2\}$ so $e_3 = (1,2)$ and $v_3=2$. We set $i=4$ and return to Step 1. 

Now $j=3$, so $v_j=v_3 = 2$. But the column $2$ contains a single dotted cell (which is $e_3$), so we must decrease $j$ by one, that is, set $j=2$ and $v_j = v_2 = 1$ and repeat Step 2. But the two dotted cells of the row $1$ are $e_2$ and $e_3$, so again we decrease $j$ by one, set $v_j = v_1 = 3$ and repeat Step 2. Once again, the two dotted cells $e_1$ and $e_2$ of the column $3$ have already appeared, so again we decrease $j$ by one, setting $v_j=v_0=0$. This time, the dotted cell $(0,5)$ isn't in $\{e_1,e_2,e_3\}$, and it is the rightmost such dotted cell in the row $0$, so we set $e_4 = (0,5)$, $v_4=5$, $i=5$ and return to Step 1, setting $j=4$. This time $v_j = v_4 = 5$, and $e_5$ is the bottommost dotted cell in the column $5$ that is not in $\{e_1,\ldots,e_4 \}$, so $e_5=(4,5)$, $v_5=4$, and we increase $i$ by one to $i=6$. Now we have reached $i=n+1$, so the algorithm terminates, producing the permutation $\tau = 31254$.

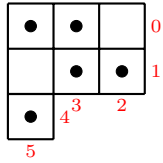
\begin{figure}[h]
  \centering
  \begin{tikzpicture}[scale=0.3]
    \draw[step=2cm,thick] (0,2) grid (6,6);
    \draw[step=2cm,thick] (0,0) grid (2,2);
    \node at (6.5,5) {\rlab{0}};
    \node at (6.5,3) {\rlab{1}};
    \node at (5,1.5) {\rlab{2}};
    \node at (3,1.5) {\rlab{3}};
    \node at (2.5,1) {\rlab{4}};
    \node at (1,-0.5) {\rlab{5}};
    \draw [fill=black] (1,1) circle [radius=0.25];
    \draw [fill=black] (1,5) circle [radius=0.25];
    \draw [fill=black] (3,3) circle [radius=0.25];
    \draw [fill=black] (3,5) circle [radius=0.25];
    \draw [fill=black] (5,3) circle [radius=0.25];
    
  \end{tikzpicture}
  \caption{An example of a \tlt\ from which Algorithm~\ref{algo-tree-tab-perm} constructs the permutation $\tau=31254$. \label{fig:ex_tlt_perm1}}
\end{figure}
\end{example}

\begin{lemma}\label{lem:well-def}
Algorithm~\ref{algo-tree-tab-perm} is well defined, and thus produces an $n$-permutation $\tau := v_1 \ldots v_n$.
\end{lemma}
\begin{proof}
We need to show that for any $i < n$ there exists a vertex $v_j$ for some $j \le i$ such that the row or column $v_j$ contains a dotted cell that isn't in $\{e_1,\ldots,e_{i-1}\}$. We proceed by contradiction. Fix $i <n$ such that for all $j \le i$, all the dotted cells of the row or column $v_j$ are in $\{e_1,\ldots,e_{i-1}\}$. Since $i <n$, there exists some dotted cell $e$ in $\clt$ that is not in $\{e_1,\ldots,e_{i-1}\}$. Without loss of generality, we may assume that $e$ has a dotted cell to its left (and none above). Let $v$ be the row containing $e$, and let $e'$ be the leftmost dotted cell in row $v$. Since $v$ can't be in $\{v_0,\ldots,v_i\}$, this means that $e'$ is not in $\{e_1,\ldots,e_{i-1}\}$.  By definition, $e'$ has no dotted cell to its left, so it must have one above. Let $e''$ be the topmost dotted cell in the column containing $e'$. As before, $e''$ cannot be in $\{e_1,\ldots,e_{i-1}\}$. By iterating these steps, we eventually reach a dotted cell $f$ in the top row that isn't in $\{e_1,\ldots,e_{i-1}\}$. But this is a contradiction, since the top row is $v_0$, and by assumption all dotted cells in $v_0$ are in $\{e_1,\ldots,e_{i-1}\}$.
\end{proof}

\begin{proposition}\label{pro:desc_rows}
Let $\tau = v_1 \ldots v_n$ be the $n$-permutation produced from a \tlt\ $\clt$ by Algorithm~\ref{algo-tree-tab-perm}. Then the descent bottoms of $\tau$ are exactly the row labels of the 
Ferrers diagram $F$ of $\clt$.
\end{proposition}

\begin{proof}
We show by induction on $i \geq 1$ that 
\begin{equation}\label{eq:desc_rows}
\mbox{$v_i$ is a descent of $\tau$ if and only if it is the label of a row of $F$.}
\end{equation} 
For $i=1$, by construction $e_1$ is the rightmost dotted cell in the top row $v_0$, and $v_1$ is the column containing $e_1$. Since $v_1$ is not a descent of $\tau$, this shows \eqref{eq:desc_rows} for $i=1$. Now fix $i \ge 1$,  and suppose that we have shown \eqref{eq:desc_rows} to hold for all $j \le i$. We say a dotted cell $e$ is \emph{included} if it is in $\{e_1,\dots,e_i\}$. We distinguish three cases.

\paragraph{Case 1} Suppose $v_i$ is a column, and there are some non-included dotted cells in $v_i$. Then $e_{i+1}$ is the bottommost of these dotted cells, and $v_{i+1}$ the row containing $e_{i+1}$. Thus $e_{i+1}$ is a cell of the Ferrers diagram $F$ in row $v_{i+1}$ and column $v_i$, so that by construction of the edge labels $v_{i+1} < v_i$, and $v_{i+1}$ is a descent, as desired.
\paragraph{Case 2} Suppose $v_i$ is a row, and there are some non-included dotted cells in~$v_i$. As above, $e_{i+1}$ is the rightmost of these cells, and it is a cell of $F$ in row $v_i$ and column $v_{i+1}$ so that $v_{i+1} > v_i$, and $v_{i+1}$ isn't a descent.
\paragraph{Case 3} Suppose all dotted cells of the row or column $v_i$ are included.
Let $e^{\rho}$ be the dotted cell in the top left corner of $\clt$. Given any dotted cell $e$, we can construct an ``edge path'' $e\rightarrow e^{\rho}$ by going from $e$ to the dotted cell immediately above it in its column or immediately to its left in its row (by definition of a \tlt\ exactly one of these exists), and iterating until we reach $e^{\rho}$.
We claim that the cell $e_{i+1}$ must be the first dotted cell on the path $e_i\rightarrow e^\rho$ that is not in $\{e_1,\ldots,e_{i-1}\}$. If $e_i$ is the first dot to be encountered with no cell below or to the right, then $e_i,\ldots,e_1$ forms the path $e_i \rightarrow e^\rho$ by definition. So $e_{i+1}$ is the first non-included dotted cell in $e_i\rightarrow e^\rho$. 

By induction suppose $e_i$ is the $k$-th cell to be encountered with no dots below and right and let $e_j$ be the $(k-1)$-th such cell, and assume the claim is true for~$e_j$. By construction $v_j,\ldots,v_{i}$ forms a path from $e_j$ to $e_{i}$, because in between these no other dotted cell is encountered with no dot below or to its right, so at each step we take the bottom or rightmost cell. Therefore, if the first non-included dotted cell we encounter, denoted $v_t$, is between $v_j$ and $v_{i}$ we take the bottommost or rightmost non-included dotted cell of $v_t$, which must be the first cell above or to the left of the previously included dotted cell of $v_t$, which implies $e_{i+1}$ is the first dot on the path from $e_i$ to $e^\rho$. If $t<j$, then by induction $e_{i+1}$ is the first non-included dotted cell on the path  $e_j\rightarrow e^\rho$, and the path $e_i\rightarrow e^\rho$ is the concatenation of $e_i\rightarrow e_j$, where we didn't encounter a non-included dotted cell, and $e_j\rightarrow e^\rho$. So $e_{i+1}$ must be the non-included dotted cell on the path $e_i\rightarrow e^\rho$.

Thus $e_{i+1}$ is on the path $e_i\rightarrow e^\rho$, which implies $e_{i+1}$ must be weakly northwest of $e_i$. Therefore, if $v_{i+1}$ is a column, then the label of $v_{i+1}$ is left of the label of $v_i$ so $v_{i+1}>v_i$, and so $v_{i+1}$ is not a descent bottom. If $v_{i+1}$ is a row, then the label of $v_{i+1}$ is above the label of $v_i$ so $v_{i+1}<v_i$, so $v_{i+1}$ is a descent bottom. This completes Case~3, and the proof.
\end{proof}

\begin{theorem}\label{thm:tree-ewt-bij}
  Let $F$ be a Ferrers diagram and let $M:\TT(F)\longrightarrow\EWT(F)$
  be the map taking $\clt$ to $\ewpmap^{-1}(\tau(\clt))$, where
  $\tau(\clt)$ is the permutation constructed from $\clt$ by
  Algorithm~\ref{algo-tree-tab-perm} and $\ewpmap$ is the map from
  Definition~\ref{phidef}. Then $M$ is a bijection from \tlts\ of shape $F$ to
  \ewts\ of shape $F$.
\end{theorem}

\begin{proof}
Proposition~\ref{pro:desc_rows}, coupled with Proposition~\ref{prop-rows-desbot}, shows that the map $M$ is well defined. To show that $M$ is a bijection, it is sufficient to show that it is injective. In fact, it suffices to show that $\clt \mapsto \tau(\clt)$ is injective, since $\ewpmap$ is a bijection by Proposition~\ref{prop-bij}. Let $\clt,\clt'$ be two \tlts\ and write $v_0,\ldots,v_n$ (resp. $v_0',\ldots,v_n'$) and $e_1,\ldots,e_n$ (resp. $e_1',\ldots,e_n'$) for the order in which vertices and edges appear in the construction of $\tau$ (resp. $\tau'$). The construction implies that for any $i \geq 1$, $(v_0,\ldots,v_i) = (v_0',\ldots,v_i')$ if and only if $(e_1,\ldots,e_i) = (e_1',\ldots,e_i')$. Now $\clt \neq \clt'$ means that the sets of dotted cells of the two tableaux are different, and in particular $(e_1,\ldots,e_n) \neq (e_1',\ldots,e_n')$. This implies that $\tau(\clt) = v_1 \cdots v_n \neq v_1' \cdots v_n' = \tau(\clt')$, as desired.
\end{proof}

\begin{example}\label{ex:TLT_EWT1}
Let $ \clt $ be the \tlt\ on the left in Figure~\ref{fig:ex_tlt_ewt1}, with rows and columns labeled.
This is the same \tlt\ as in Example~\ref{ex:TLT_perm1} and in Figure~\ref{fig:ex_tlt_perm1}, and the corresponding permutation is $\tau=31254$. The \ewt\ on the right is constructed by first filling the top row with 1s, then reading $\tau$ from left to right, and filling in the remaining cells of $v_i$ with 1s if $v_i$ is the label of a row, and with 0s if it is the label of a column.

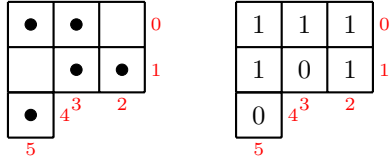
\begin{figure}[h]
  \centering
  \begin{tikzpicture}[scale=0.3]
    \draw[step=2cm,thick] (0,2) grid (6,6);
    \draw[step=2cm,thick] (0,0) grid (2,2);
    \node at (6.5,5) {\rlab{0}};
    \node at (6.5,3) {\rlab{1}};
    \node at (5,1.5) {\rlab{2}};
    \node at (3,1.5) {\rlab{3}};
    \node at (2.5,1) {\rlab{4}};
    \node at (1,-0.5) {\rlab{5}};
    \draw [fill=black] (1,1) circle [radius=0.25];
    \draw [fill=black] (1,5) circle [radius=0.25];
    \draw [fill=black] (3,3) circle [radius=0.25];
    \draw [fill=black] (3,5) circle [radius=0.25];
    \draw [fill=black] (5,3) circle [radius=0.25];

    \def\x{10}
    \draw[step=2cm,thick] (0+\x,2) grid (6+\x,6);
    \draw[step=2cm,thick] (0+\x,0) grid (2+\x,2);
    \node at (6.5+\x,5) {\rlab{0}};
    \node at (6.5+\x,3) {\rlab{1}};
    \node at (5+\x,1.5) {\rlab{2}};
    \node at (3+\x,1.5) {\rlab{3}};
    \node at (2.5+\x,1) {\rlab{4}};
    \node at (1+\x,-0.5) {\rlab{5}};
    \node at (1+\x,1) {$0$};
    \node at (3+\x,3) {$0$};
    \node at (1+\x,3) {$1$};
    \node at (1+\x,5) {$1$};
    \node at (3+\x,5) {$1$};
    \node at (5+\x,3) {$1$};
    \node at (5+\x,5) {$1$};
  \end{tikzpicture}
  \caption{An example of a \tlt\ $ \clt $ (left) and its corresponding \ewt\ $\clt '$ (right) via the bijection of Theorem~\ref{thm:tree-ewt-bij}, both corresponding to the permutation $\tau=31254$ \label{fig:ex_tlt_ewt1}.}
\end{figure}
\end{example}

\begin{remark}
The map in Theorem~\ref{thm:tree-ewt-bij} can be combined with the results of Theorems~\ref{thm:EWT_ASM} and \ref{thm:tree-tt} to obtain a bijection between spanning trees of the Ferrers graph $G=G(F)$ and minimal recurrent configurations of the ASM on~$G$. In fact, this bijection is equivalent to, but presented in a different form from, the bijection obtained by composing those of Lemmas 2 and 4 in~\cite{CLB}, restricted to spanning trees with no external activity, which map bijectively to minimal recurrent configurations. Indeed, it is possible to show that the construction 
in Algorithm~\ref{algo-tree-tab-perm} yields the following property. Let $<_E$ be any $F$-compatible total order on the edges of $G$. Then for all $i \geq 0$, $e_{i+1}$ is the largest edge, according to $<_E$, where the corresponding cell in $ \clt $ lies in exactly one column or row of $\{v_0,\ldots,v_i\}$, that is, 
$$e_{i+1} = \max\limits_{<_E} \{ (w,w') \in S: w \notin \{v_0,\ldots,v_i\}, w' \in \{v_0,\ldots,v_i\} \}.$$
\end{remark}

\begin{remark}\label{rem:diff_bij1}
The bijection from Theorem~\ref{thm:tree-ewt-bij} is different from the bijection in Section~\ref{sec-bij-new-tree}. Consider the \tlt\ on the left in Figure~\ref{fig:ex_diff_bij1}. The permutation $\tau$ obtained by Lemma~\ref{lem:well-def} is $\tau = 4231$, which leads to the \ewt\ in the middle of Figure~\ref{fig:ex_diff_bij1}. However, the \ewt\ obtained through the bijection of Section~\ref{sec-bij-new-tree} is the one on the right, and these are indeed different \ewts.
\end{remark}

\begin{figure}[h]
  \centering
  \begin{tikzpicture}[scale=0.3]
    \draw[step=2cm,thick] (0,0) grid (4,6);
    \node at (4.5,5) {\rlab{0}};
    \node at (4.5,3) {\rlab{1}};
    \node at (4.5,1) {\rlab{2}};
    \node at (3,-0.5) {\rlab{3}};
    \node at (1,-0.5) {\rlab{4}};
    \draw [fill=black] (1,1) circle [radius=0.25];
    \draw [fill=black] (1,5) circle [radius=0.25];
    \draw [fill=black] (1,3) circle [radius=0.25];
    \draw [fill=black] (3,1) circle [radius=0.25];
    
    \def\x{8}
    \draw[step=2cm,thick] (0+\x,0) grid (4+\x,6);
    \node at (4.5+\x,5) {\rlab{0}};
    \node at (4.5+\x,3) {\rlab{1}};
    \node at (4.5+\x,1) {\rlab{2}};
    \node at (3+\x,-0.5) {\rlab{3}};
    \node at (1+\x,-0.5) {\rlab{4}};
    \node at (1+\x,1) {$0$};
    \node at (3+\x,3) {$0$};
    \node at (1+\x,3) {$0$};
    \node at (3+\x,1) {$1$};
    \node at (1+\x,5) {$1$};
    \node at (3+\x,5) {$1$};
    \node at (2+\x,-2) {$4231$};
    
    \def\y{16}
    \draw[step=2cm,thick] (0+\y,0) grid (4+\y,6);
    \node at (4.5+\y,5) {\rlab{0}};
    \node at (4.5+\y,3) {\rlab{1}};
    \node at (4.5+\y,1) {\rlab{2}};
    \node at (3+\y,-0.5) {\rlab{3}};
    \node at (1+\y,-0.5) {\rlab{4}};
    \node at (1+\y,1) {$0$};
    \node at (3+\y,3) {$1$};
    \node at (1+\y,3) {$0$};
    \node at (3+\y,1) {$1$};
    \node at (1+\y,5) {$1$};
    \node at (3+\y,5) {$1$};
    \node at (2+\y,-2) {$4213$};
  \end{tikzpicture}
  \caption{An example of a \tlt\ $ \clt $ (left), its corresponding \ewt\ $\clt'$ and permutation via the bijection of Theorem~\ref{thm:tree-ewt-bij} (center), and the corresponding \ewt\ and permutation via the bijection of Section~\ref{sec-bij-new-tree}.
The two \ewts\ are different. \label{fig:ex_diff_bij1}}
\end{figure}
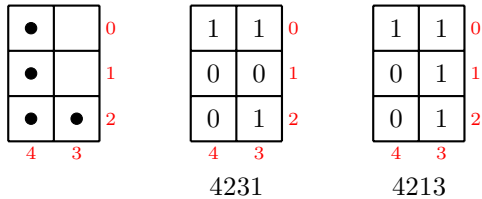

\begin{question}
Although the \ewts\ in Figure~\ref{fig:ex_diff_bij1} are different, their corresponding permutations have a common prefix. In Example~\ref{ex:TLT_EWT1} we get the same \ewts\ and thus the same permutation for both bijections, namely 31254.  It is easy to see that the first letter of the corresponding permutation  is always the same for both bijections; the question is what determines the maximal common prefix. In particular, can we find conditions under which the two bijections give the same \ewts\ (and permutations)?
\end{question}

\section{Statistics and properties of \ew- and \newts\ and their corresponding permutations}\label{sect-stats}

In this section we collect various results on \ew- and \newts, and their corresponding permutations, and statistics on these.  We also mention two open problems and one conjecture that we have been unable to resolve.

\begin{proposition}\label{ew-fixed} There is a one-to-one correspondence between all-1 columns in a \ew- or \newt\ $\clt$ and fixed points in the permutation $\desexc(\ewpmap(\clt))$.
\end{proposition}

\begin{proof} Given an all-1 column $C$, we need to show that when constructing the permutation $\pi=\ewpmap(\clt)$ the label $\ell$ of column $C$ is read after all rows and columns with smaller labels have been read and that the label read just before $\ell$ is smaller than $\ell$.  These two conditions imply, according to the definition of the map $\desexc$ in Proposition~\ref{de-bij}, that $\ell$ will be in place $\ell$ in $\desexc(\pi)$.

  The (label of a) column $C$ is read only after all its 1s have been deleted, so if $C$ is an all-1 column, all the rows with smaller labels will have been read before $C$ is read.  When that has been done, since columns are read from right to left, next all unread columns to the right of $C$ will be read, since they will have been cleared of 1s, just as $C$.  Thus all smaller labels will be read before $\ell$, and one of them will be the last one read before $\ell$.

Conversely, if $\ell$ is a fixed point of $\desexc(\ewpmap(\clt))$, then the label read just before $\ell$ in $\pi=\ewpmap(\clt)$ is smaller than $\ell$, so that $\ell$ is the label of a column $C$. Moreover, all row labels that are smaller than $\ell$ are read before $\ell$ in $\pi$, so that by Corollary~\ref{invbij} the column $C$ is all 1s.
\end{proof}

A  \emph{decreasing adjacency} in a permutation is a pair of adjacent letters $a_ia_{i+1}$ such that $a_i=a_{i+1}+1$. For example, 65478132 has three decreasing adjacencies, 65, 54 and 32. 
\begin{lemma}\label{fix-adj}
  The number of fixed points different from 1 has the same distribution on $n$-permutations as does the number of decreasing adjacencies.
\end{lemma}
\begin{proof}
  The bijection $\desexc$ translates a decreasing adjacency to a ``minimal excedance'', that is, an excedance whose top is one greater than its bottom.  A cyclic right shift of the resulting permutation $\pi$ takes each such minimal excedance to a fixed point.  All fixed points in $\pi$ arise this way, except 1, if 1 is a fixed point.
\end{proof}

\begin{proposition}\label{all0rows} The distribution of the number of all-zero-rows in \ewts\ of size  $n$ is the same as the distribution of fixed points different from 1 in $n$-permutations, and the same as the distribution of decreasing adjacencies on such permutations.
\end{proposition}

\begin{proof}
An \az\ row in a \ewt\ $\clt$ corresponds to an \ao\ column in the \newt\ $\cls=\tr(\clt)$, where $\tr$ is the diagonal reflection of $\clt$ described before Definition~\ref{def-new}. The \ao\ columns in $\cls$ correspond one-to-one to fixed points in 
$\desexc(\ewpmap(\cls))$ except that 1 is a fixed point if and only if the top row of $\cls$, and no other row, is all 1s. The latter part of the claim therefore follows from Lemma~\ref{fix-adj}.
\end{proof}

\begin{proposition}\label{paths} The number of vertices in a longest directed path in a unique-sink acyclic orientation of a Ferrers graph $\clf$ corresponding to a \ewt\ $\clt$ equals the number of blocks in the run-decomposition of\/ $\pi=\ewpmap(\clt)$.  Moreover, any directed path in such an orientation contains at most one vertex label from each such block in $\pi$.
\end{proposition}
\begin{proof}
  As pointed out just after the proof of Lemma~\ref{4-cycle}, a \ewt\ $\clt$ corresponds
uniquely to an acyclic orientation (with a unique sink) of the corresponding Ferrers graph $\clf$.  More precisely, a 1 in column $c$ and row $r$ corresponds to an edge from the vertex labeled $c$ to the vertex labeled $r$ in $\clf$, a 0 to that edge being oriented the other way.  A directed path through $\clf$  must alternate between ``top'' and ``bottom'' vertices in the bipartite graph $\clf$, corresponding, respectively, to row and column labels of $\clt$, and the contents of the cells in~$\clt$ corresponding to successive row  and column labels must alternate between 0 and 1.  

It follows from Lemma~\ref{desbot-lemma} that each block in the run-decomposition of $\pi=\ewpmap(\clt)$ contains only row labels or only column labels, and so a directed path in $\clf$ must alternate between column and row label blocks in $\pi$.  We claim that if there is an edge from a vertex $u$ to a vertex $v$ then $v$ must precede $u$ in $\pi$, which implies that a path can have at most one vertex from each block. 
Indeed, if $u$ labels a row and there is an edge from $u$ to $v$, then $v$ is a column and the entry in the cell $(u,v)$ is $0$, so that by Corollary~\ref{invbij} $v$ is read before $u$ in $\pi$. An analogous argument shows that if $u$ is a column label then the row label $v$ must have been read before $u$.

We claim that the sequence consisting of the rightmost letter in each block in the run-decomposition of $\pi$ gives the labels of vertices constituting a directed path in $\clf$, when we read that sequence backwards, which will complete the proof.  

When a column label is read, a 0 is cleared from all rows with a cell in that column whose labels are yet to be read; when a row label is read, a 1 is cleared from all columns with a cell in that row whose labels are yet to be read.  The rightmost label in a block of column labels is the largest of them; the rightmost label in a block of row labels is the smallest of them.  Thus, if $c$ is the last label of a column in a block and $r$ the last label of a row in the following block, then $c>r$, which implies that there is a cell in column $c$ and row $r$, and that cell must have a 0, since $c$ is read before $r$,
implying that there is an edge from vertex $r$ to vertex $c$ in $\clf$.  An analogous argument shows that $r<c'$, where $c'$ is the rightmost label in next block after that of $r$, and that there is a cell with a 1 in column $c'$ and row~$r$.
\end{proof}

\begin{proposition}\label{one0fib} The number of \ewts\ of size $n$ with exactly one 0 in each non-top row is the Fibonacci number $F_{2n-2}$.  More precisely, the number of such tableaux whose first row has length $k$ is $\binom{2n-1-k}{k-1}$.
\end{proposition}
\begin{proof}  Call a \ewt\ 0-minimal if it has just one zero in each non-top row. Such a tableau is completely determined by the letters in each maximal set of adjacent cells at the bottom of columns in the same row, at most one of which letters can be a 0, and exactly one of which must be 0 in the case of the bottom row. 
This is because the cells directly above such a maximal set of cells can only have a 0 if there is a 0 in the bottom cell of the same column, since otherwise we would violate condition \ref{ewt3} in Definition \ref{def:ewt}. Therefore, every cell above a $1$ must be filled with a $1$ and every cell above a $0$ must be filled with a $0$ until a row that already contains a $0$ is encountered and then everything from that row up must contain a $1$.

Given a filling of each of the sets of maximal adjacent cells mentioned above, with at most one 0 in each set and exactly one 0 in the bottom row, we construct a binary string determined by the orientations of the edges on the NE-SW border of the tableaux and the content of their adjacent cells, starting from the edge just after the vertical edge of the top row.  If the first edge is horizontal we write 01 for each edge in the maximal initial sequence of horizontal edges.  For any other set of adjacent horizontal edges we write, starting from the right, 10 for an edge if there is a 1 in the cell above it, until we see a 0 (if at all), at which point we write 01.  For each of the remaining edges in the same maximal set of edges (that are at the same level) we write 01.  For each vertical edge we write 00. Note that we read the edges of the tableau from right to left, but construct the string from left to right.

It is easy to see that the binary string thus constructed begins with a 0, ends with 01 and has no adjacent 1s.  It is also easy to verify the following.  Given a binary string $w$ of length $2n-3$ with no adjacent 1s, let $\widehat{w}=0w01$.  Then $\widehat{w}$ is the binary string, of length $2n$, corresponding, as described above, to a unique \ewt\ of size $n$ with a single 0 in each non-top row.  Moreover, if $\widehat{w}$ has $k$ occurrences of 1 (so $w$ has $k-1$ occurrences of 1)  then $\widehat{w}$ corresponds to such a tableau with top row of length $k$.  It is easy to show that the number of binary strings of length $2n-3$ with $k-1$ occurrences of 1, none of them adjacent, is $\binom{2n-1-k}{k-1}$, and it is well known that the sum of those binomial coefficients, over all integers $k$, is the Fibonacci number $F_{2n-2}$, if $F_0=F_1=1$.
\end{proof}

A permutation $\pi$ is said to \emph{contain the pattern $p$}, where $p$ is also a permutation, if $\pi$ has a subsequence of letters that appear in the same order of size as the letters of $p$, and $\pi$ is said to \emph{avoid $p$} if $\pi$ has no such occurrence.  For example, $\pi=351624$ contains the pattern 231 in the subsequence 562, but $\pi$ avoids 321, as it has no decreasing subsequence of length 3.

\begin{proposition}\label{pro:231_pattern}
Let $\clt$ be a \ewt\ with at least two rows, and $\pi=\newpmap(\clt)$ the corresponding
permutation. If $\pi$ avoids $231$, then the second row of $ \clt $ has a unique~$0$.
\end{proposition}

\begin{proof}
If $\clt$ has a second row, we first show that any column with a $0$ in the second row of $\clt$ must have all entries equal to $0$ except in the top row. Let $i$ be the label of the second row, and $j$ the label of a column with a $0$ in the second row. Then $j$ appears before $i$ in the permutation $\pi$ by construction. Since $j>i$ and $\pi$ avoids $231$, this implies in particular that there cannot be any row label~$i'$ with $i' < j$ that appears before $j$ in $\pi$. Thus, all entries in the column labeled $j$ (other than the top row) must be a $0$, since these correspond to such labels $i'$, which must appear after $j$ in $\pi$.

Suppose now that there are two columns with labels $j<j'$ whose entries in the second row are both $0$. By the above, both columns must have all non-top entries equal to $0$. By construction, such columns are read in increasing order at the start of the permutation $\pi$. In particular, $j$ appears before $j'$, and both appear before $i$ (since their entries in the row labeled $i$ are $0$). But since $i<j<j'$, it follows that $j$, $j'$ and $i$ appearing in that order in $\pi$ form a $231$ pattern, which is forbidden.
\end{proof}
\begin{question}
  Proposition~\ref{pro:231_pattern} is not an equivalence; there are \ewts\ with a unique 0 in the second row whose corresponding permutations do not avoid 231, such as the $2\times2$ tableau with a single 0, in the bottom right corner, whose corresponding permutation is in fact 231 itself.  Is there some nice further condition on \ewts\ that would make this an equivalence?
\end{question}

\begin{proposition}\label{pro:213_pattern}
Let $\clt$ be a \ewt, and  $\pi=\newpmap(\clt)$ the corresponding permutation. Then the following are equivalent.
\begin{enumerate}
\item The permutation $\pi$ avoids the pattern $213$.
\item 	
  \begin{itemize}
  \item All rows of $\clt$ are of the form $0 \cdots 0 1 \cdots 1$, where the sequence of $1$s may be empty, and
  \item in any row of $\clt$, the leftmost $1$ has no $1$ in its column below it.
  \end{itemize}
\end{enumerate}
\end{proposition}
\begin{proof}
We first show that $(1)$ implies $(2)$. Suppose $\pi$ avoids $213$. First,
fix some row of $\clt$, labeled $i$. Let $j$ be the label of the column of the leftmost $1$ in that row (if it exists). By construction we have $j>i$ and since the entry $(i,j)$ is a~$1$, $i$ appears before $j$ in $\pi$. Let $j'$ be the label of any column to the right of the column labeled $j$. Since $i<j'<j$ and $\pi$ avoids $213$, $j'$ must also appear after~$i$ in $\pi$, and thus the entry $(i,j')$ must be a $1$. This shows the first point of~$(2)$.

For the second point, we need to show that all entries in the column labeled~$j$ that are below the row labeled $i$ are $0$. Let $i'$ be the label of the row of such an entry. By construction we have $i<i'<j$. Moreover, recall that $i$ appears before $j$ in $\pi$. Since $\pi$ avoids $213$, $i'$ must appear after $i$ in $\pi$. Now by the above, the rows labeled $i$ and $i'$ are both of the form $0 \cdots 0 1 \cdots 1$. By construction, if the leftmost $1$ in the row $i'$ was to the left (or directly below) the leftmost $1$ in the row labeled $i$, then $i'$ would appear before $i$ in $\pi$, which is not the case. In particular, the entry directly below the leftmost $1$ in the row labeled $i$ must be a $0$, as desired.

We now show that $(2) \Rightarrow (1)$. Suppose that $\clt$ satisfies both points of $(2)$. We first show the following lemma.

\begin{lemma}\label{lem:213_pattern_1_is_row}
If $\pi$ contains a $213$-pattern, then it contains a $213$-pattern where the letter
corresponding to the $1$ of the pattern is the label of a row of $\clt$.
\end{lemma}
\begin{proof}
Let $b,a,c$ be the letters of an occurrence of 213 corresponding, respectively, to $2,1,3$.  Suppose $a$ is the label of a column (otherwise we are done).  The~$b$ cannot be the label of a column in the same block in the run-decomposition of~$\pi$, since such blocks are increasing. Let $a'$ be the  leftmost letter in the block of $a$, so $a'\le a$.  By Lemma~\ref{run-decomp}, the rightmost letter $x$ of the block of row labels preceding $a'$ is smaller than $a'$, and thus $b,x,c$ form the desired occurrence of~213.
\end{proof}

Seeking contradiction, suppose $\pi$ contains an occurrence of $213$, and write $i,j,k$ for the letters corresponding respectively to $1,2$ and $3$. By Lemma~\ref{lem:213_pattern_1_is_row}, we may assume that $i$ is the label of a row of $\clt$. We distinguish three cases, where $(i,j)$ refers to the cell in row $i$ and column $j$ of $\clt$.

\begin{itemize}
\item[Case 1.] Suppose $j$ and $k$ both label columns of $\clt$. Since $i<j<k$, row $i$ has entries in both those columns, and column $j$ is to the right of column $k$. Moreover, since $j,i,k$ appear in $\pi$ in that order, the entry in $(i,j)$ is a $0$ and the entry in $(i,k)$ is a $1$. This contradicts the fact that each row is of the form $0 \ldots 0 1 \ldots 1$.

\item[Case 2.] Suppose $j$ labels a column and $k$ labels a row of $\clt$. Since $i<j<k$, column $j$ must have an entry in row $i$ but not in row $k$. Since $j$ appears before $i$ in $\pi$, the entry in $(i,j)$ is a $0$. Because of the form of row $i$, all entries to the left of $(i,j)$ must also be $0$. In particular, all entries of row~$i$  in columns that also contain entries in row $k$ are $0$, which implies that $k$ appears before $i$ in $\pi$. This contradicts the assumption that $j,i,k$ appear in $\pi$ in that order. 

\item[Case 3.] Suppose $j$ labels a row of $\clt$. Because $i<j$, row $j$ is below row $i$ in $\clt$. 
Now the conditions of $(2)$, together with the fact that $j$ appears before $i$ in $\pi$, imply that the entry of all the columns in row $i$ that also have an entry in row $j$ are 0s. This implies that $k$ labels a row, since any column label greater than $j$ must therefore appear before $i$ in $\pi$. But $j<k$, so that all the columns in row $i$ that have an entry in row $j$ also have an entry in row $k$. Since these entries are all 0s, and $i$ appears before $k$ in $\pi$, this implies that the row $k$ is all-0. But such a row would be read before~$i$, since the column of the leftmost 1 of row $i$ has no entry in row $k$. This is the desired contradiction.
\end{itemize}
\kern-5.25ex
\end{proof}

\begin{definition}
  A row $R$ in a tableau $\clt$ \emph{dominates} a row $S$ in $\clt$ if $S$
lies above $R$ and there is no column in $\clt$ where $R$ has a $0$ and $S$ has a $1$.  A tableau~$\clt$ is \emph{domination-free} if no row of $\clt$ dominates another.
\end{definition}

\begin{definition}
  A tableau $\clt$ is \emph{top-justified} if all the 1s in each column are at the top, that is, no 1 has a 0 above it in the same column.
\end{definition}

\begin{lemma}\label{dom-top-just}
  A domination-free tableau $\clt$ is top-justified.
\end{lemma}
\begin{proof}
Note first that no two rows of $\clt$ can have the same filling or the lower one would dominate the upper row. Given a row $u$ above a row $r$ in $\clt$, the filling of the cells in $r$ can of course be obtained from the filling of $u$ by changing some entries. We cannot only change some 0s to 1s, since then $r$ would dominate $u$, so we must change some 1s in $u$ to 0s in $r$. If we then change some 0s in $u$ to 1s in $r$ we would get the kind of pattern forbidden by the definition of \newts.  Therefore, any cell containing a $1$ must have only $1$s above it. 
\end{proof}

Recall the definition of a pattern in a permutation, given at the beginning of this section.  A permutation $\pi$ avoids the \emph{vincular pattern} 32--1 if $\pi$ has no decreasing subsequence of length 3 whose first two letters are adjacent.  For example, 4132 avoids 32--1, whereas 4231 does not, because of the subsequence 421. Recall also that a \emph{right-to-left minimum} of a permutation $\pi$ is a letter of $\pi$ that is smaller than all letters to its right.  For example, 3172546 has right-to-left minima $1,2,4,6$.

\begin{proposition}
  A domination-free  \newt\ $\cln$ of size $n$ has exactly $k$ columns containing a 0 if and only if the permutation $\pi=\newpmap(\cln)$ begins with 1, avoids the vincular pattern $32$--$1$ and has $(n-k)$ right-to-left minima.  Thus, such \newts\ are counted by \mbox{$S(n-1,n-1-k)$}, the Stirling number of the second kind, that is, the number of partitions of a set of $(n-1)$ elements into $(n-1-k)$ blocks.
\end{proposition}
\begin{proof}
  By Lemma~\ref{dom-top-just}, $\cln$ is top-justified, so every column has a cell that contains a $1$ with only $1$s above and no 1s below, and we call such a cell a \emph{transition cell} if it has some 0 below it (that is, is not the bottom cell in its column).  This implies the top row is all $1$s and that every non-bottom row must contain a transition cell, since a 1 must be lost in some column going from a row to the next one below.  We let $\ell$ denote the number of rows and $z$ the number of columns with no zeros, so $n-k=\ell+z$.

First we show the forward direction. Given a \newt\ $\clt$ with no dominating rows, the first letter of $\pi=\newpmap(\clt)$ must be a $1$ because  only the top row is all $1$s. Next note that on a pass from bottom to top we will never find two rows that have no $0$s, since otherwise the lower row would dominate the upper one. Also note that a row cannot be only $1$s unless every row above is only $1$s, which means all the rows above must already have been emptied. Moreover, if all rows above are empty, then all columns with a smaller label must be empty as well. So on each pass from bottom to top we encounter exactly one row that we can empty, which gives a letter that is a right-to-left minimum. As these are the only elements that can be the bottom of a descent this implies that $\pi$ avoids $32$--$1$. So every row contributes a right-to-left minimum.  
In addition, the label of an all-$1$ column appears in the block immediately after the label of its bottom row in the run-decomposition of $\pi$, and since column labels are read in increasing order, the all-$1$ columns contribute $z$ extra right-to-left minima (see the end of the previous paragraph).

For the other direction, suppose $\pi=\newpmap(\clt)$ avoids $32$--$1$, has $t$ 
right-to-left minima and begins with $1$. Let $s,r$ be two row labels of $\clt$ such that row $s$ is above row $r$. Then $r \neq 1$ so $r$ is not the first letter of $\pi$ and is therefore a descent bottom by point 2 of Theorem~\ref{props-new}. Since $\pi$ avoids $32$--$1$, this implies in particular that $r$ must appear after $s$ in $\pi$. Let $c$ be the column label immediately preceding $r$. By construction, $c > r > s$ and $s,c,r$ appear in $\pi$ in that order, so the column $c$ has a $0$ in row $r$ and a $1$ in row $s$, so that row $r$ does not dominate row $s$. Thus $\clt$ is domination-free. Moreover, we have shown that the labels of rows appear in $\pi$ in order from top to bottom, so that each is a right-to-left minimum of $\pi$. As above, any all-$1$ column is also a right-to-left minimum. Finally, a column $c$ containing a $0$ in row $r$ must be read before $r$ in $\pi$, and since $c>r$, $c$ is not a left-to-right minimum. Thus the number of columns containing a $0$ is $n-t$ as desired.

Finally, the fact that the permutations described in the statement of the theorem are counted by the Stirling numbers of the second kind was shown in~\cite[Porism~1]{cla-gpa}.
\end{proof}

\begin{proposition}\label{top-ones}
There is a bijection between top-justified \newts\ of size $n$ with~$k$ rows
and set partitions of $[n]=\{1,\ldots,n\}$ with $k$ blocks.  The number of such \newts\ thus equals $S(n,k)$, the Stirling number of the second kind.  
\end{proposition}

\begin{proof}
Let $\mathsf{N}(n,k)$ be the set of top-justified \newts\ of size $n$ with $k$ rows, and $\mathsf{P}(n,k)$ the set of set partitions of $[n]$ with $k$ blocks. Given $\cln \in \mathsf{N}(n,k)$, we construct a set partition $F(\cln)=\clp_1,\ldots,\clp_k \in \mathsf{P}(n,k)$ as follows. 
Let $r_1,\ldots,r_k$ be the row labels of $\cln$. For all
$j\in\{1,\ldots,k\}$, we define $\clp_j$ to be the set consisting of $r_j$ and of all column labels $c$ such that the bottommost 1 of the column labeled $c$ is in the row labeled $r_j$. We claim that $F: \mathsf{N}(n,k) \rightarrow \mathsf{P}(n,k)$ is a bijection.

Each column has exactly one bottommost 1, so that each column label belongs to exactly one block of $\clp$, as does each row label by construction. Thus $\clp$ is a set partition of $[n]$ with $k$ blocks, and $F$ is well defined. To see that $F$ is a bijection, we construct its inverse. 

Let $\clp=\clp_1,\ldots,\clp_k\in\mathsf{P}(n,k)$. For all $j\in\{1,\ldots,k\}$, let $r_j$ be the minimum element of the block $\clp_j$. Let $\clf$ be the Ferrers diagram whose row labels are $r_1,\ldots,r_k$.  We define a filling $\cln=G(\clp)$ of $\clf$ as follows.
Any element $c\not=r_j$ in block $\clp_j$ is greater than $r_j$ and labels a column of $\clf$, and thus row $r_j$ has a cell in column~$c$.  Fill that cell and all the cells above it in column $c$ with 1s, and all the cells below that cell in column $c$ with a 0.  This clearly defines a filling where every column has a 1 and no 1 has a 0 below it in the same column.

Moreover, no such filling of a Ferrers diagram can contain a rectangle violating condition~2 of Definition~\ref{def-new}, since there cannot be a 0 above a 1 in a column. Thus $\cln \in \mathsf{N}(n,k)$ as desired, and it is easy to see that the maps $F$ and $G$ are inverses of each other.
\end{proof}

\begin{proposition}\label{pro:231-newts}
  If the permutation $\pi=\newpmap(\cln)$ for a \newt\ $\cln$ 
avoids the pattern 231 then $\cln$ has no 0 above a 1 in a column.
\end{proposition}

\begin{proof}
Suppose $\cln$ has a column labeled $k$ that has a 0 in a row labeled $i$ above a~1 in a row labeled $j$. By construction $i<j<k$, and from  point~3 of Theorem~\ref{props-new} the labels $j,k,i$ appear in that order in $\pi$, which forms a 231 pattern.
\end{proof}

Propositions~\ref{top-ones} and \ref{pro:231-newts} suggest the following problem, since $n$-permutations avoiding 231 are equinumerous with non-crossing\footnote{A set partition has a crossing if it has a block containing $a$ and $c$ and another block containing $b$ and $d$, where $a<b<c<d$.} set partitions of $[n]$ (both counted by the Catalan numbers):
\begin{problem}
Find a bijection from the \newts\ in Proposition~\ref{top-ones} to set partitions, such that the set partition corresponding to a \newt\ $\cln$ is non-crossing if and only if the permutation  $\pi=\newpmap(\cln)$ avoids 231.
\end{problem}

Note that the bijection $F$ from the proof of Proposition~\ref{top-ones} does not solve this problem. Indeed, consider the \newt\ $\cln$ with two rows and two columns, whose entries read by rows are $[1,1],[1,0]$. We have $\cln \in \mathsf{N}(4,2)$. However, the permutation $\pi=\newpmap(\cln)$ is given by $1324$, which avoids 231, whereas the set partition $\clp=F(\cln)$ is given by $\{1,3\},\{2,4\}$, which has a crossing.

A \emph{big descent} in a permutation $\pi=a_1a_2\ldots a_n$ is an index $i$ such that $a_i\ge a_{i+1}+2$.

\begin{proposition}\label{pro:newts-0-column}
  The number of \newts\ of size $n$ with $k$ columns containing a 0 equals the number of $n$-permutations with $k$ big descents.
\end{proposition}

\begin{proof}
Proposition~\ref{all0rows} implies that the distribution of the number of rows containing at least one 1 in \ewts\ of size $n$ is the same as the distribution of big descents on $n$-permutations, because decreasing adjacencies are precisely those descents that aren't big. But in the construction at the beginning of Section~\ref{sec-newt}, \newts\ of size $n$ with $k$ columns containing a 0 correspond exactly to \ewts\ of size $n$ with $k$ rows containing a 1, which proves the claim.
\end{proof}

It follows from Theorem~\ref{numtab} (and already from \cite[Cor.~4.5]{ew-ferrers}) that \ewts\ of size $n$ with 
$k+1$ rows are equinumerous with $n$-permutations with~$k$ excedances.  It follows, of course, that the number of such tableaux with $k$ columns equals the number of $n$-permutations with $k$ non-excedances.  We have however not been able to prove the following innocent-looking conjecture, which we have verified for $n\le12$.
\begin{conjecture}
  The number of \ewts\ of size $n$ with $k$ columns containing a 0 equals the Eulerian number $A(n,k)$ counting $n$-permutations with $k$ excedances.
\end{conjecture}

\bibliographystyle{abbrv}
\bibliography{ewtableaux}

\end{document}